\documentclass[11pt,a4paper,twoside,bibliography=totoc]{scrartcl}
\usepackage{a4wide}
\usepackage{amsfonts}
\usepackage{amsmath}
\usepackage{amssymb}
\usepackage{array}
\usepackage{amsthm}
\usepackage[utf8]{inputenc}
\usepackage{subcaption}
\usepackage{xifthen}
\usepackage{graphicx}
\usepackage{color}
\usepackage{pgf}
\usepackage{scrpage2}
\usepackage{multirow}
\usepackage{multicol}
\usepackage{listings} 
\usepackage{tikz}
\usetikzlibrary{math}
\usetikzlibrary{patterns}
\usepackage{graphicx}
\usepackage{wrapfig}

\newcommand{\N}{\ensuremath{\mathbb{N}}}
\newcommand{\T}{\ensuremath{\mathbb{T}}}

\newcommand{\B}{\ensuremath{\mathbb{B}}}

\newcommand{\Z}{\ensuremath{\mathbb{Z}}}

\newcommand{\R}{\ensuremath{\mathbb{R}}}
\newcommand{\C}{\ensuremath{\mathbb{C}}}

\newcommand{\ii}{\textnormal{i}}

\newcommand{\eip}[1]{\textnormal{e}^{2\pi\ii{#1}}}
\newcommand{\eim}[1]{\textnormal{e}^{-2\pi\ii{#1}}}
\newcommand{\norm}[1]{\left\Vert #1\right\Vert}

\newcommand{\floor}[1]{\left\lfloor#1\right\rfloor}
\newcommand{\ceil}[1]{\left\lceil#1\right\rceil}

\newcommand{\pmat}[1]{\begin{pmatrix} #1 \end{pmatrix}}

\newcommand{\set}[1]{\left\{ #1 \right\}}
\newcommand{\inter}[1]{\left( #1 \right)}
\newcommand{\interc}[1]{\left[ #1 \right]}
\newcommand{\intercl}[1]{\left[ #1 \right)}

\newcommand{\abs}[1]{\left| #1 \right |}

\newcommand{\dirich}[1]{\ifthenelse{\isempty{#1}}
	{D_n}
	{D_n\left(#1\right)} }		
\newcommand{\dirichd}[1]{\ifthenelse{\isempty{#1}}
	{D_n'}
	{D_n'\left(#1\right)} }	
\newcommand{\dirichdd}[1]{\ifthenelse{\isempty{#1}}
	{D_n''}
	{D_n''\left(#1\right)} }	
\newcommand{\dirichm}[1]{\ifthenelse{\isempty{#1}}
	{\tilde D_n}
	{\tilde D_n\left(#1\right)} }

\DeclareMathOperator*{\lmax}{\lambda_{\max}}
\DeclareMathOperator*{\lmin}{\lambda_{\min}}
\DeclareMathOperator*{\smax}{\sigma_{\max}}
\DeclareMathOperator*{\smin}{\sigma_{\min}}

\DeclareMathOperator*{\cond}{cond}

\newtheorem{thm}{Theorem}[section]
\newtheorem{lemma}[thm]{Lemma}
\newtheorem{remark}[thm]{Remark}
\newtheorem{definition}[thm]{Definition}
\newtheorem{example}[thm]{Example}
\newtheorem{corollary}[thm]{Corollary}
\newtheorem{proposition}[thm]{Proposition}

\numberwithin{equation}{section}
\numberwithin{table}{section}
\numberwithin{figure}{section}

\newcommand{\bend}{\hspace*{0ex} \hfill \hbox{\vrule height
    1.5ex\vbox{\hrule width 1.4ex \vskip 1.4ex\hrule  width 1.4ex}\vrule
    height 1.5ex}}

\long\def\symbolfootnote[#1]#2{\begingroup%
\def\thefootnote{\fnsymbol{footnote}}\footnote[#1]{#2}\endgroup}

\renewcommand{\mathbf}[1]{\ensuremath{\boldsymbol{#1}}}

\renewcommand{\thefootnote}{\fnsymbol{footnote}}

\allowdisplaybreaks
\title{On the condition number of Vandermonde matrices with pairs of nearly-colliding nodes}

\date{\today}

\author{Stefan Kunis\footnotemark[1] \qquad Dominik Nagel\footnotemark[1]}

\newif\ifshow
\showtrue

\begin{document}

\maketitle

\begin{abstract}
We prove upper and lower bounds for the spectral condition number of rectangular Vandermonde matrices with nodes on the complex unit circle.
The nodes are ``off the grid'', pairs of nodes nearly collide, and the studied condition number grows linearly with the inverse separation distance.
Such growth rates are known in greater generality if all nodes collide or for groups of colliding nodes.
For pairs of nodes, we provide reasonable sharp constants that are independent of the number of nodes as long as non-colliding nodes are well-separated.
\medskip

\medskip

\end{abstract}

\footnotetext[1]{
  Osnabr\"uck University, Institute of Mathematics
  \texttt{\{skunis,dnagel\}@uos.de}
}

\section{Introduction}  

Vandermonde matrices with complex nodes appear in polynomial interpolation problems and many other fields of mathematics, see e.g.~the introduction of \cite{AuBo17} and its references.
In this paper, we are interested in rectangular Vandermonde matrices with nodes on the complex unit circle and with a large polynomial degree.
These matrices generalize the classical discrete Fourier matrices to non-equispaced nodes and the involved polynomial degree is also called bandwidth.
The condition number of those matrices has recently become important in the context of stability analysis of super-resolution algorithms like Prony's method \cite{Pr95,KuMoPeOh17}, the matrix pencil method \cite{HuSa90,Mo15}, the ESPRIT algorithm \cite{RoKa89,PoTa2017}, and the MUSIC algorithm \cite{Sc86,LiFan2016}.
If the nodes of such a Vandermonde matrix are all well-separated, with minimal separation distance greater than the inverse bandwidth, bounds on the condition number are established for example in \cite{Ba99,KuPo07,Mo15,AuBo17}.

If nodes are nearly-colliding, i.e.~their distance is smaller than the inverse bandwidth, the behavior of the condition number is not yet fully understood.
The seminal paper \cite{Do92} coined the term (inverse) super-resolution factor for the product of the bandwidth and the separation distance of the nodes.
For $M$ nodes on a grid, the results in \cite{Do92,DeNg15} imply that the condition number grows like the super-resolution factor raised to the power of $M-1$ if \emph{all} nodes nearly collide.
More recently, the practically relevant situation of groups of nearly-colliding nodes was studied in \cite{CaMo16,GoYo17,LiLi17,BaDeGoYo18}.
In different setups and oversimplifying a bit, all of these refinements are able to replace the exponent $M-1$ by the smaller number $m-1$, where $m$ denotes the number of nodes that are in the largest group of nearly-colliding nodes.
The authors of \cite{CaMo16,GoYo17} focus on quite specific quantities in an optimization approach and in the so-called Prony mapping, respectively.
In contrast, the condition number or the relevant smallest singular value of Vandermonde matrices with ``off the grid'' nodes on the unit circle is studied in \cite{LiLi17,BaDeGoYo18}.
While \cite{BaDeGoYo18} provided the exponent $m-1$ for the first time, the proof technique leads to quite pessimistic constants and more restrictively asks all nodes (including the well-separated ones) to be within a
tiny arc of the unit circle.
More recently, the second version of \cite{LiLi17} provided a quite general framework and reasonable sharp constants, but involves a technical condition which prevents the separation distance from going to zero for a fixed number of nodes and a fixed bandwidth.

Here we present upper and lower bounds for the condition number of Vandermonde matrices with pairs of nearly-colliding nodes, i.e., the special case $m=2$.
We achieve the expected linear order and all constants are reasonable sharp and absolute.
In contrast to the more general quoted results \cite{BaDeGoYo18,LiLi17}, the nodes can be placed on the full unit circle and the separation distance is allowed to approach zero.
Our mild technical conditions, which seem to be artifacts of our proof technique, are i) a logarithmic growth in the separation distance of the well-separated nodes, ii) a uniformity condition that colliding nodes behave similarly, and iii) an a-priori upper bound on the separation distance of the colliding nodes. 

The outline of this paper is as follows:
Section 2 fixes the notation, recalls results for the case of well-separated nodes, and provides lower bounds for the condition number.
In Section 3, we establish upper bounds for nodes that are well-separated from each other except for one pair of nodes that is nearly-colliding.
Section 4 goes one step further and studies the more general case where an arbitrary number of pairs of nodes nearly collide.
Theoretical and numerical comparison with \cite{BaDeGoYo18,LiLi17} can be found at the end of Section 4 and in Section 5, respectively.

\section{Preliminaries}\label{ch:preliminaries}  
Let $\T:=\set{z\in\C\colon \abs{z}=1}$ be the complex torus and nodes $\set{z_1,\dots,z_M} \subset \T$ be parametrized by $z_j= \eim{t_j}, j=1\dots,M$, such that $t_1< \cdots< t_M \in \intercl{0,1}$.
We fix a degree $n \in \N$ so that $N:=2n+1 > M$ and set up the rectangular Vandermonde matrix
\begin{align}\label{eq:vandermonde}
	A:= \pmat{z_j^k}_{\substack{j=1,\dots,M \\ \abs{k}\le n}} = \pmat{z_1^{-n}& \cdots & z_1^{-1} & 1 & z_1^1 & \cdots & z_1^n\\ \vdots & & \vdots & & \vdots \\ z_M^{-n}& \cdots & z_M^{-1} & 1 & z_M^1 & \cdots & z_M^n} \in \C^{M\times N}.
\end{align}
The Dirichlet kernel $D_n \colon \R\rightarrow \R$ is given by
\begin{equation}\label{eq:Dn}
 D_n(t):= \sum_{k=-n}^{n}\eip{kt} =
  \begin{cases}
    N, & t\in \Z,\\
    \frac{\sin(N\pi t)}{\sin(\pi t)}, & \text{otherwise},
  \end{cases}
\end{equation}
so that
\begin{equation*}
K := AA^*= \pmat{\dirich{t_i-t_j}}_{i,j=1}^M \in \R^{M\times M}.
\end{equation*}
The matrix $K$ is symmetric positive definite and the spectral condition number
\begin{equation*}
	\cond(A) := \frac{\smax(A)}{\smin(A)}=\sqrt{\norm{K}\norm{K^{-1}}}
\end{equation*}
is finite since all nodes are distinct (here and throughout the paper $\|K\|:=\sup\{\|Kx\|:\;\|x\|=1\}$ with $\|x\|^2:=\sum_{k}|x_k|^2$).
On the other hand, if two nodes are equal, then two rows of $A$ are the same and by continuity the condition number diverges if two nodes collide.
The (wrap around) distance of two nodes is given by
\begin{equation*}
	\abs{t_j-t_\ell}_\T:= \min_{r\in \Z}\abs{t_j-t_\ell+r}.
\end{equation*}
and we introduce the \emph{normalized separation distance} of the node set as
\begin{equation*}
	\tau := N \min_{j\ne\ell}\abs{t_j-t_\ell}_\T.
\end{equation*}
We call the case $\tau=1$ \emph{critical separation}, i.e.~$\min_{j\ne\ell}\abs{t_j-t_\ell}_\T = \frac{1}{N}$, and
the cases $\tau \le 1$ and $\tau>1$ \emph{nearly-colliding} and \emph{well-separated}, respectively. 
Figure \ref{fig:4nodes} illustrates the situation for 4 nodes on the unit circle. The parameter $\rho_{min}$ describes a minimum separation distance of involved non-colliding nodes assumed in the Theorems.

\begin{figure}[h!]
\centering
	\begin{subfigure}{0.32\textwidth}
	 \begin{tikzpicture}[scale=2.5]
	 \tiny
	 \tikzmath{\x1 = .1; \x2 =.25; \x3 = .3; \x4 = 1; \y1=\x4-\x1; \y2=\x4-\x2; \y3=\x4-\x3;} 
	 \draw [draw=black,pattern=dots,pattern color=black] (\x3,0) rectangle (\y3,\x2);
	 \draw [draw=black,pattern=dots,pattern color=black] (0,\x3) rectangle (\x2,\y3);
	 \draw [draw=black,pattern=north east lines,pattern color=black] (0,\y2)--(0,\x4)--(\x2,\y2)--(0,\y2);
	 \draw [draw=black,pattern=north east lines,pattern color=black] (\y2,0)--(\x4,0)--(\y2,\x2)--(\y2,0);
	 \draw [draw=black,fill=gray] (\x2,\x2)--(\y2,\x2)--(\x2,\y2)--(\x2,\x2);
	 \draw (\x2,-.02) -- (\x2,.02) node[below=3pt] {$1$};
	 \draw (\x3,-.02) -- (\x3,.02) node[below=10pt,right=-4pt] {$\rho_{\min}$};
	 \draw (\x4,-.02) -- (\x4,.02) node[below=4pt] {$\frac{N}{2}$};
	 \draw (-.02,\x2) -- (.02,\x2) node[left=2pt] {$1$};
	 \draw (-.04,\x3) -- (.02,\x3) node[left=8pt] {$\rho_{\min}$};
	 \draw (-.02,\x4) -- (.02,\x4) node[left=4pt] {$\frac{N}{2}$};
	 \draw (\y2,-.02) -- (\y2,.02) node[above=4pt] {};
	 \draw (\y3,-.02) -- (\y3,.02) node[above=4pt] {};
	 \draw (-.02,\y2) -- (.02,\y2) node[right=4pt] {};
	 \draw (-.02,\y3) -- (.02,\y3) node[right=4pt] {};
	 \draw (0,\x4) -- (\x4,0);
	 \draw[->,thick] (0,0) -- (1.1,0) node[above=4pt, right=-8pt] {$\tau_{1,2}$};
	 \draw[->,thick] (0,0) -- (0,1.1) node[above=-3pt, right=1pt] {$\tau_{3,4}$};
	 \end{tikzpicture}
		\caption{Theorem \ref{thm:2collup}}
	\end{subfigure}
	\begin{subfigure}{0.32\textwidth}
	 \begin{tikzpicture}[scale=2.5]
	 \tiny
		\tikzmath{\x1 = .15; \x2 =.25; \x3 = .35; \x4 = 1; \y1=\x4-\x1; \y2=\x4-\x2; \y3=\x4-\x3;} 
		\draw [draw=black,pattern=dots,pattern color=black](0,0) rectangle (\x1,\x1);
		\draw [draw=black,pattern=dots,pattern color=black] (\x3,0) rectangle (\y3,\x1);
		\draw [draw=black,pattern=dots,pattern color=black] (0,\x3) rectangle (\x1,\y3);
		\draw [draw=black,pattern=north east lines,pattern color=black] (\y2,0)--(\x4,0)--(\y2,\x2)--(\y2,0);
		\draw [draw=black,pattern=north east lines,pattern color=black] (0,\y2)--(0,\x4)--(\x2,\y2)--(0,\y2);
		\draw [draw=black,fill=gray] (\x2,\x2) -- (\y2,\x2) -- (\x2,\y2) -- (\x2,\x2);
		\draw (\x1,-.02) -- (\x1,.02) node[below=10pt,left=-6pt] {$\tau_{max}$}; 
		\draw (\x2,-.02) -- (\x2,.02) node[below=3pt] {$1$};
		\draw (\x3,-.02) -- (\x3,.02) node[below=10pt,right=-4pt] {$\rho_{\min}$};
		\draw (\x4,-.02) -- (\x4,.02) node[below=4pt] {$\frac{N}{2}$};
		\draw (-.04,\x1) -- (.02,\x1) node[left=8pt] {$\tau_{\max}$};
		\draw (-.02,\x2) -- (.02,\x2) node[left=2pt] {$1$};
		\draw (-.04,\x3) -- (.02,\x3) node[left=8pt] {$\rho_{\min}$};
		\draw (-.02,\x4) -- (.02,\x4) node[left=4pt] {$\frac{N}{2}$};
		\draw (\y1,-.02) -- (\y1,.02) node[above=4pt] {};
		\draw (\y2,-.02) -- (\y2,.02) node[above=4pt] {};
		\draw (\y3,-.02) -- (\y3,.02) node[above=4pt] {};
		\draw (-.02,\y1) -- (.02,\y1) node[right=4pt] {};
		\draw (-.02,\y2) -- (.02,\y2) node[right=4pt] {};
		\draw (-.02,\y3) -- (.02,\y3) node[right=4pt] {};
		\draw (0,\x4) -- (\x4,0);
		\draw[->,thick] (0,0) -- (1.1,0) node[above=4pt, right=-8pt] {$\tau_{1,2}$};
		\draw[->,thick] (0,0) -- (0,1.1) node[above=-3pt, right=1pt] {$\tau_{3,4}$};
		\end{tikzpicture}
		\caption{Theorem \ref{thm:pairwUp}}
	\end{subfigure}
	\begin{subfigure}{0.32\textwidth}
	 \begin{tikzpicture}[scale=2.5]
	 \tiny
		\tikzmath{\x1 = .15; \x2 =.25; \x3 = .36; \x4 = 1; \y1=\x4-\x1; \y2=\x4-\x2; \y3=\x4-\x3;} 
		\node at (0.2*\x2/10,0.2*\x2/10) [circle,fill,inner sep=0.4pt]{};
		\node at (\x2/10,\x2/10) [circle,fill,inner sep=0.4pt]{};
		\node at (2*\x2/10,2*\x2/10) [circle,fill,inner sep=0.4pt]{};
		\node at (3*\x2/10,3*\x2/10) [circle,fill,inner sep=0.4pt]{};
		\node at (4*\x2/10,4*\x2/10) [circle,fill,inner sep=0.4pt]{};
		\node at (5*\x2/10,5*\x2/10) [circle,fill,inner sep=0.4pt]{};
		\node at (6*\x2/10,6*\x2/10) [circle,fill,inner sep=0.4pt]{};
		\node at (7*\x2/10,7*\x2/10) [circle,fill,inner sep=0.4pt]{};
		\node at (8*\x2/10,8*\x2/10) [circle,fill,inner sep=0.4pt]{};
		\node at (9*\x2/10,9*\x2/10) [circle,fill,inner sep=0.4pt]{};
		\node at (10*\x2/10,10*\x2/10) [circle,fill,inner sep=0.4pt]{};
		\draw [draw=black,pattern=dots,pattern color=black] (\x3,0) rectangle (\y3,\x2);
		\draw [draw=black,pattern=dots,pattern color=black] (0,\x3) rectangle (\x2,\y3);
		\draw [draw=black,pattern=north east lines,pattern color=black] (0,\y2)--(0,\x4)--(\x2,\y2)--(0,\y2);
		\draw [draw=black,pattern=north east lines,pattern color=black] (\y2,0)--(\x4,0)--(\y2,\x2)--(\y2,0);
		\draw [draw=black,fill=gray] (\x2,\x2)--(\y2,\x2)--(\x2,\y2)--(\x2,\x2);
		\draw (\x2,-.02) -- (\x2,.02) node[below=3pt] {$1$};
		\draw (\x3,-.02) -- (\x3,.02) node[below=10pt,right=-4pt] {$\rho_{\min}$};
		\draw (\x4,-.02) -- (\x4,.02) node[below=4pt] {$\frac{N}{2}$};
		\draw (-.02,\x2) -- (.02,\x2) node[left=2pt] {$1$};
		\draw (-.04,\x3) -- (.02,\x3) node[left=8pt] {$\rho_{\min}$};
		\draw (-.02,\x4) -- (.02,\x4) node[left=4pt] {$\frac{N}{2}$}; 
		\draw (\y2,-.02) -- (\y2,.02) node[above=4pt] {};
		\draw (\y3,-.02) -- (\y3,.02) node[above=4pt] {};
		\draw (-.02,\y2) -- (.02,\y2) node[right=4pt] {};
		\draw (-.02,\y3) -- (.02,\y3) node[right=4pt] {};
		\draw (0,\x4) -- (\x4,0);
		\draw[->,thick] (0,0) -- (1.1,0) node[above=4pt, right=-8pt] {$\tau_{1,2}$};
		\draw[->,thick] (0,0) -- (0,1.1) node[above=-3pt, right=1pt] {$\tau_{3,4}$};
		\end{tikzpicture}
		\caption{Theorem \ref{thm:pairwUpc1}}
	\end{subfigure}
	\caption{Sketch of four-node configurations, $t_1< t_2< t_3< t_4\in \intercl{0,1}$, $t_1=0$, $t_3=1/2$, $N$ large enough, $\tau_{1,2}:=N\abs{t_1-t_2}_\T$, $\tau_{3,4}:=N\abs{t_3-t_4}_\T$. dotted: Theorem can be applied, filled: well-separated, lined: 3 nearly-colliding nodes, empty areas: at most 2 nearly-coll.~nodes, but not covered by results.}\label{fig:4nodes}
\end{figure}
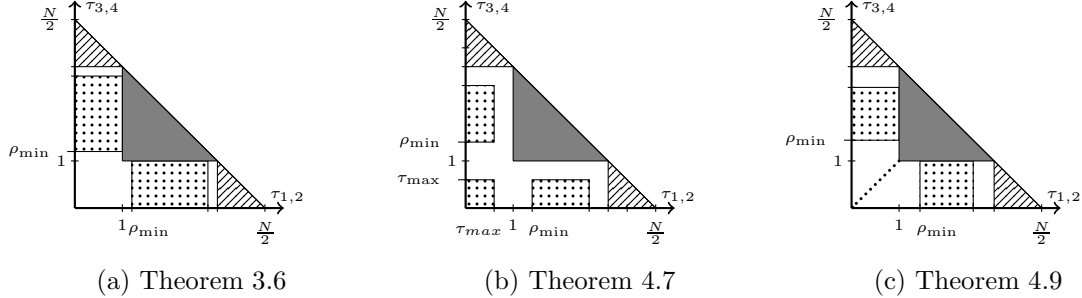

A reasonable result for well-separated nodes is as follows.
\begin{thm}[\cite{Mo15,AuBo17}]\label{la:condWell}
	Let $A$ be a Vandermonde matrix as in \eqref{eq:vandermonde} with $\tau>1$, then
	\begin{equation*}
	  N\left(1-\frac{1}{\tau}\right) \le \smin^2(A)\le N \le \smax^2(A)\le N\left(1+\frac{1}{\tau}\right).
	\end{equation*}
	In particular, we have
	\begin{equation*}
         \cond(A)^2 \le 1+ \frac{2}{\tau-1}
        \end{equation*}
	and this implies $\norm{K} \le N + N/\tau$ and $\norm{K^{-1}}=\norm{A^\dagger}^2\le (N - N/\tau)^{-1}$, where $A^\dagger := A^*(AA^*)^{-1}$ denotes the Moore-Penrose pseudo inverse of $A$.
\end{thm} 

We note in passing, that the above lower bound on the smallest singular value is an improvement of \cite{Mo15} by \cite{AuBo17} and that \cite{Mo15} and \cite{Di19} allow to replace $\frac{1}{\tau}$ in the upper and the lower bound by $\frac{1}{\tau}-\frac{1}{N}$, respectively.
Moreover, we have the following lower bound on the condition number.
This already shows that the upper bound for well-separated nodes is quite sharp and provides the benchmark for nearly-colliding nodes.

\begin{thm}[Lower bound]\label{thm:condlowerbound}
 Let $A$ be a Vandermonde matrix as in \eqref{eq:vandermonde}, then
 \begin{equation*}
  \smin^2(A) \le N-\abs{D_n(\tau/N)} \le N \le N+\abs{D_n(\tau/N)}\le\smax^2(A).
 \end{equation*}
 In particular, we have
  \begin{equation*}
    \cond(A)^2 \ge 1+\frac{2}{\pi\tau-1}
  \end{equation*}
 for $\tau\in\N+\frac{1}{2}$, uniformly in $N$ and almost matching the above upper bound.
 
 For nearly-colliding nodes, we have
  \begin{equation*}
    \cond(A)^2 \ge \frac{12}{\pi^2\tau^2}-1 \ge \frac{1}{\tau^2}
  \end{equation*}
 for $\tau\le \sqrt{12/\pi^2-1}\approx 0.46$ and $\cond(A) \ge \sqrt{6}/\pi\tau\approx 0.77/\tau$ for all $\tau\le 1$.
\end{thm}
\begin{proof}
 Without loss of generality, let $t_2-t_1=\tau/N$ and consider the upper left $2\times 2$-block in
 \begin{equation*}
  K= \pmat{C & * \\ * & *}, \quad C:=\pmat{\dirich{0} & \dirich{\tau/N}\\ \dirich{\tau/N}& \dirich{0}}.
 \end{equation*}
 We apply Lemma $\ref{interlacing}$, get 
 \begin{equation*}
  \cond(A)^2=\frac{\lmax(K)}{\lmin(K)}\ge \frac{\lmax(C)}{\lmin(C)} = \frac{\dirich{0}+\abs{\dirich{\tau/N}}}{\dirich{0}-\abs{\dirich{\tau/N}}} = 1+\frac{2\abs{\dirich{\tau/N}}}{N-\abs{\dirich{\tau/N}}},
 \end{equation*}  
 and Lemma \ref{la:dirichbounds} yields the assertion.
\end{proof}

\section{Nodes with one nearly-colliding pair}\label{ch:2coll}
\begin{definition}\label{def:2coll}
Let $M\ge 2$ and $0=t_1<\cdots<t_M \in \intercl{0,1}$ such that
\begin{align*}
\abs{t_1-t_2}_\T &= \frac{\tau}{N}, &0&<\tau\le 1,\\
\abs{t_j-t_\ell}_\T &\ge \frac{\rho}{N}, \;j\ne \ell,\;\ell\ge 3, &1&<\rho<\infty,
\end{align*}
then $\{t_1,\hdots,t_M\}$ is called a set of \emph{nodes with one nearly-colliding pair}, see Figure \ref{fig:2coll} for an illustration.
Due to periodicity, the choice $t_1=0$ and $\abs{t_1-t_2}_\T = \frac{\tau}{N}$ is without loss of generality.
\end{definition}
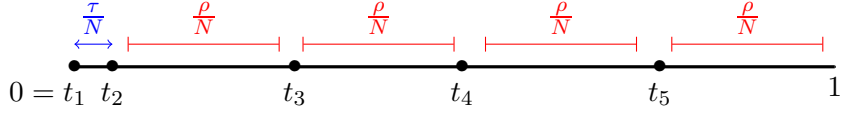
\begin{figure}[h!]
	\begin{center}
		\begin{tikzpicture}[xscale=10]
		\tikzmath{\r = 0.2; \t =0.05; \a=0.3;
			\x3=\t+\r*1.2; \d3=0.5*(\x3-\t-\r);
			\x4=\x3+\r*1.1; \d4=0.5*(\x4-\x3-\r);
			\x5=\x4+\r*1.3; \d5=0.5*(\x5-\x4-\r);\d6=0.5*(1-\x5-\r);}
		\draw[-][draw=black, very thick] (0,0) -- (1,0);
		\draw [thick] 	(1,0)		node[below]{1} -- (1,0);
		\node at (0,0) [label={[below,shift={(-0.35,-0.95*\a)}] {$0=t_1$}}]{$\bullet$};
		\node at (\t,0) [label={[below,shift={(0,-\a)}] {$t_2$}}]{$\bullet$};
		\draw [<->,blue](0,\a) 	-- (\t,\a) node[midway,above]{$\frac{\tau}{N}$};
		\node at (\x3,0) [label={[below,shift={(0,-\a)}] {$t_3$}}]{$\bullet$};
		\draw [|-|,red] (\t+\d3,\a) -- (\x3-\d3,\a) node[midway,above]{$\frac{\rho}{N}$};
		\node at (\x4,0) [label={[below,shift={(0,-\a)}] {$t_4$}}]{$\bullet$};
		\draw [|-|,red] (\x3+\d4,\a) -- (\x4-\d4,\a) node[midway,above]{$\frac{\rho}{N}$};
		\node at (\x5,0) [label={[below,shift={(0,-\a)}] {$t_5$}}]{$\bullet$};
		\draw [|-|,red] (\x4+\d5,\a) -- (\x5-\d5,\a) node[midway,above]{$\frac{\rho}{N}$};
		\draw [|-|,red] (\x5+\d6,\a) -- (1-\d6,\a) node[midway,above]{$\frac{\rho}{N}$};
		\end{tikzpicture}
	\end{center}
	\caption{Example of a node set with $M=5$ satisfying Def.~\ref{def:2coll}.\label{fig:2coll}}
\end{figure}

Now, we estimate an upper bound on the condition number of the Hermitian matrix $K$ by bounding $\norm{K}$ directly and applying Lemma \ref{la:schurDecomp} to $K^{-1}$ before bounding $\norm{K^{-1}}$. 
For that, we introduce some notation for abbreviation.
\begin{definition}\label{def:2coll2}
	We define $a_1 := \pmat{z_1^k}_{\abs{k}\le n}\in \C^{1\times N}$ and $A_2:= \pmat{z_j^k}_{\substack{j=2,\dots,M\\ \abs{k}\le n}}\in \C^{(M-1)\times N}$ so that with
	\begin{align}\label{eq:decomp11}
		a_1a_1^*=N, \quad K_2:= A_2A_2^* \quad\text{and}\quad b:=A_2 a_1^*=\pmat{D_n(\tau/N)\\D_n(t_3)\\ \vdots\\ D_n(t_M)},
	\end{align}
	we have the partitioning
	\begin{align}\label{eq:decomp12}
		A=\pmat{a_1\\ A_2} \quad \text{and} \quad K =\pmat{N & b^* \\ b & K_2},
	\end{align}
	 where $A_2$ is a Vandermonde matrix with nodes that are at least $\frac{\rho}{N}$ separated.
\end{definition}

\begin{lemma}\label{la:2collnormK}
	Under the conditions of Definition \ref{def:2coll} and for $\rho\ge 6$, we have
	\begin{equation*}
	\norm{K}\le 2.3 N.
	\end{equation*}
\end{lemma}
\begin{proof}
The key idea is to see the set of nodes as a union of two well-separated subsets and use the existing bounds for these. In contrast to the next chapter, here, one of the sets only consist of a single node.
We start by noting that Theorem \ref{la:condWell} and \eqref{eq:decomp11} yield $\norm{b}^2\le \norm{a_1}^2 \norm{A_2}^2=N\norm{K_2}$.
Together with the decomposition \eqref{eq:decomp12}, the triangle inequality, Lemma \ref{blockGersch}, and Theorem \ref{la:condWell}, we obtain
\begin{equation*}
\norm{K} \le \norm{\pmat{N&0\\0&K_2}} +\norm{\pmat{0&b^*\\b&0}}\le\norm{K_2}+\norm{b}
\le N\left(\frac{\rho+1}{\rho}+\sqrt{\frac{\rho+1}{\rho}}\right).
\end{equation*}
\end{proof}

\begin{lemma}\label{la:btaylor}
Under the conditions of Definition \ref{def:2coll} and with $b$ as in $\eqref{eq:decomp11}$, we have
	\begin{equation*}
	b=K_2 e_1 +r,
	\end{equation*}
	where $e_1\in \R^{(M-1)}$ denotes the first unit vector and
	\begin{equation*}
	\norm{r}^2 \le \left(N-\dirich{\tau/N}\right)^2 +N^2\tau^2 \left(\frac{\pi^4}{12\rho^2}+\frac{1.21\pi}{\rho^3}+ \frac{\pi^4}{180\rho^4} \right).
	\end{equation*}
\end{lemma}
\begin{proof}
 The vector $b$ can be approximated by the first column of $K_2$ in the sense that
  \begin{equation*}
	b=\pmat{\dirich{\tau/N}\\ \dirich{t_3}\\ \vdots\\ \dirich{t_M}}=\pmat{\dirich{0}\\ \dirich{t_3-\tau/N}\\ \vdots\\ \dirich{t_M-\tau/N}}+ \pmat{r_1 \\ \vdots\\ r_{M-1}}.
  \end{equation*}
	We have $\abs{r_1} = N-\dirich{\tau/N}$ and for $j=2,\hdots,M-1$ the mean value theorem yields
	\begin{equation*}
		\abs{r_j}=\abs{\dirich{t_{j+1}}-\dirich{t_{j+1}-\tau/N}}=\abs{\dirichd{\xi_j}}\frac{\tau}{N}, \quad \xi_j\in (\abs{t_{j+1}-\frac{\tau}{N}}_\T,\abs{t_{j+1}}_\T).
	\end{equation*}
	Note that, in the worst case, half of the nodes can be as close as possible (under the assumed separation condition) to $t_2$ not only on its right but also on its left.
	Hence, for $j=2,\dots,\ceil{\frac{M}{2}}$, $\xi_j \ge \frac{(j-1)\rho}{N}$ and Lemma \ref{la:dirichbounds} lead to
	\begin{equation*}
	\abs{r_j}\le N\left(\frac{\pi}{2N\abs{\xi_j}}+\frac{1}{2N^2\abs{\xi_j}^2}\right)\tau
	\le N\left(\frac{\pi}{2(j-1)\rho}+\frac{1}{2(j-1)^2\rho^2}\right)\tau.
	\end{equation*}
	Thus, for all nodes, we get 
	\begin{equation*}
		\sum_{j=2}^{M-1}|r_j|^2 
		\le 2 \sum_{j=2}^{\lceil M/2\rceil}|r_j|^2
		\le N^2\tau^2 \left( 
		\frac{\pi^2}{2\rho^2} \underbrace{\sum_{j=1}^{\infty}\frac{1}{j^2}}_{=\frac{\pi^2}{6}}
		+ \frac{\pi}{\rho^3} \underbrace{\sum_{j=1}^{\infty} \frac{1}{j^3}}_{\le 1.21} 
		+ \frac{1}{2\rho^4}\underbrace{\sum_{j=1}^{\infty} \frac{1}{j^4}}_{=\frac{\pi^4}{90}}
		\right).
	\end{equation*}
\end{proof}

\begin{lemma}\label{la:2collnormKinv}
Under the conditions of Definition \ref{def:2coll} and for $\rho \ge 5$, we have
	\begin{equation*}
	\norm{K^{-1}}\le \frac{C(\rho)}{N\tau^2},
	\end{equation*}
	where
	\begin{equation*}
	C(\rho) = \left(\frac{2\rho-1}{\rho-1}+\sqrt{\frac{\rho}{\rho-1}} \right) \left[2- \frac{\rho}{\rho-1}  \left(1+\frac{\pi^4}{12\rho^2}+\frac{1.21\pi}{\rho^3}+\frac{\pi^4}{180\rho^4}  \right)\right]^{-1}.
	\end{equation*}
\end{lemma}
\begin{proof}
	We consider $K$ decomposed as in $\eqref{eq:decomp12}$ and apply Lemma \ref{la:schurDecomp} with respect to $K_2$ to obtain
	\begin{equation*}
	K^{-1} = \pmat{I&0\\-K_2^{-1}b&I} \pmat{(N-b^*K_2^{-1}b)^{-1}& 0\\0&K_2^{-1}} \pmat{I & -b^*K_2^{-1} \\ 0 & I}
	\end{equation*}
	and thus,
	\begin{equation*}
		\norm{K^{-1}} \le \norm{\pmat{I&0\\-K_2^{-1}b&I}}^2 \max\left\{\norm{K_2^{-1}},\norm{\left(N-b^*K_2^{-1}b\right)^{-1}}\right\}.
	\end{equation*}

	First of all, we establish an upper bound for the norm of the triangular matrix.
	Equation \eqref{eq:decomp11} and Theorem \ref{la:condWell} imply
	\begin{equation*}
		 \norm{K_2^{-1}b}=\norm{(A_2 A_2^*)^{-1} A_2 a_1^*}\le \norm{A_2^\dagger}\norm{a_1}\le\sqrt{\frac{\rho}{\rho-1}}.
	\end{equation*}
	Together with Lemma \ref{blockGersch}, we obtain
	\begin{align}\label{eq:2collTriang}
		\norm{\pmat{I&0\\-K_2^{-1}b&I}}^2
		\le 1+\norm{K_2^{-1}b}+\norm{K_2^{-1}b}^2
		\le \frac{2\rho-1}{\rho-1}+\sqrt{\frac{\rho}{\rho-1}}.
	\end{align}
	The next step is to bound $(N-b^*K_2^{-1}b)^{-1}$. Lemma \ref{la:btaylor} yields
	\begin{equation*}
		b^*K_2^{-1}b = (K_2 e_1+ r)^* K_2^{-1} (K_2 e_1+ r)
		= 2 D_n(\tau/N)-D_n(0)+r^*K_2^{-1}r.
	\end{equation*}
	Applying the second part of Lemma \ref{la:btaylor}, Lemma \ref{la:dirichbounds}, and Theorem \ref{la:condWell} yields
	\begin{align*}
		N-b^*K_2^{-1}b \ge& 2\left(N-D_n(\tau/N)\right) - \norm{r}^2\norm{K_2^{-1}}\\
		\ge& \left(N-D_n(\tau/N)\right) \left( 2 - \left(N-D_n(\tau/N)\right)\norm{K_2^{-1}}\right)- \norm{K_2^{-1}} \sum_{j=2}^{M-1}\abs{r_j}^2 \\
		\ge& N\tau^2 \left(2-N\norm{K_2^{-1}}\right) -\|K_2^{-1}\| N^2\tau^2 \left(\frac{\pi^4}{12\rho^2}+\frac{1.21\pi}{\rho^3}+\frac{\pi^4}{180\rho^4}\right) \\
		\ge&  N\tau^2 \left[2- \frac{\rho}{\rho-1}  \left(1+\frac{\pi^4}{12\rho^2}+\frac{1.21\pi}{\rho^3}+\frac{\pi^4}{180\rho^4}  \right)\right].
	\end{align*}
	For $\rho\ge 5$, the most inner bracketed term takes values in $(1,1.4)$ such that the square bracketed term is positive. Forming the reciprocal gives the result, since
	Theorem \ref{la:condWell} also implies
	\begin{align}\label{eq:2collSchurMax}
		 N\norm{K_2^{-1}}\le \frac{\rho}{\rho-1}\le \frac{\rho-1}{\rho-2}\le \left[2- \frac{\rho}{\rho-1}  \left(1+\hdots\right)\right]^{-1}.
	\end{align}
\end{proof}

\begin{thm}[Upper bound]\label{thm:2collup}
Under the conditions of Definition \ref{def:2coll} with $\rho \ge \rho_{\min}=6$, we have
	\begin{equation*}
	\cond(A)\le \frac{4}{\tau}.
	\end{equation*}
\end{thm}
\begin{proof}
	The bound follows from Lemmata \ref{la:2collnormK} and \ref{la:2collnormKinv} with $C(\rho)\le C(6) \le 6.5$.
\end{proof}

\begin{wrapfigure}{r}{0.35\textwidth}
 \centering
    \includegraphics[width=0.3\textwidth]{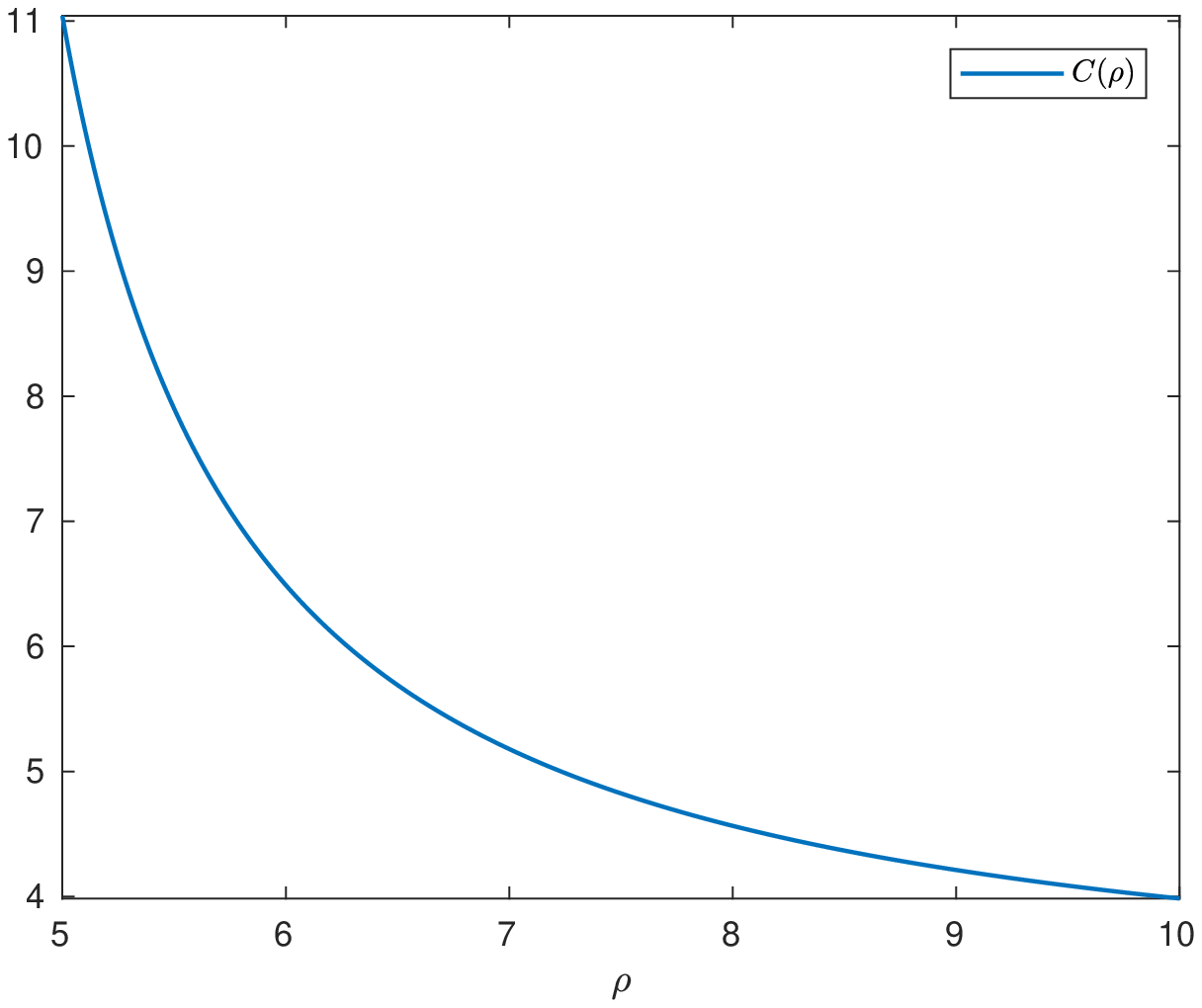}
    \vspace{-0.35cm}
  \caption{$C(\rho)$ in Lem.~\ref{la:2collnormKinv}.}
  \label{fig:2coll_C}
  \vspace{-1cm}
\end{wrapfigure}

 Lower and upper bounds in Theorems \ref{thm:condlowerbound} and \ref{thm:2collup} yield
 \begin{equation*}
  \frac{1}{\tau}\le\cond(A)\le \frac{4}{\tau}
 \end{equation*}
 for $\tau\le 0.46$ and $6\le\rho$.
 The condition on $\rho$ implies that for specific configurations of $M$ nodes, our result becomes effective as early as $N\approx 6M$ - this is in contrast to the results \cite{BaDeGoYo18,LiLi17}, where $N$ has to be much larger.

 \begin{remark}[Constants]\label{rm:2coll}
 Some comments regarding what is lost during our proof:
 \begin{enumerate}
 \item The constant in Lemma \ref{la:2collnormK} is a numerical value for all $\rho\ge 6$, indeed the proof is valid for all values $\rho>1$.
 The case $M=2$ shows that Lemmata \ref{la:2collnormK} and \ref{la:btaylor} are reasonable sharp since in this case $\norm{K}=N+D_n(\tau/N)\ge N(2-\pi^2\tau^2/6)$ and $\norm{r}=N-D_n(\tau/N)\ge N(2-\tau^2)$, see Lemma \ref{la:dirichbounds} for the two inequalities.
 \item In Lemma \ref{la:2collnormKinv}, the constant $C(\rho)$ is monotone decreasing in $\rho$, see also Figure \ref{fig:2coll_C}.
 It is bounded below by 3 which is due to the relatively crude norm estimate on the block triangular factors in the Schur complement decomposition. 
 Note that the left hand side in \eqref{eq:2collTriang} is bounded from below by $1+\norm{K_2^{-1}b}^2$. 
 An additional minor improvement on $C(\rho)$ and on the range of admissible values for $\rho$ can be achieved when applying Lemma \ref{la:dirichbounds} to two factors simultaneously.
 \end{enumerate}
\end{remark}

\begin{remark}[Generalizations and limitations]
 In principle, the suggested Schur complement technique can be generalized to more than two nodes colliding and also to the multivariate case:
 \begin{enumerate}
  \item Let $M\ge 3$ and $0=t_1<\cdots<t_M \in \intercl{0,1}$ be such that $\{t_1,t_2,t_3\}$ nearly-collide and decompose
 \begin{align*}
   K =\pmat{K_1 & B^* \\ B & K_2},\quad
   K_1=\pmat{N & \dirich{t_1-t_2}\\ \dirich{t_1-t_2} & N},\quad
   K_2=\pmat{\dirich{t_i-t_j}}_{i,j=3}^M.
 \end{align*}
 While it is clear that the Schur complement $K_1-B^* K_2^{-1} B$ is strictly positive definite, establishing a lower bound on its smallest singular value similar to the proofs of Lemmata \ref{la:btaylor} and \ref{la:2collnormKinv} seems considerably harder. Already the linear approximation in Lemma \ref{la:btaylor} then needs to be replaced by a higher order approximation for the matrix $B$.

 \item Consider the bivariate case and the Vandermonde matrix
   \begin{equation*}
    A= \pmat{z_j^\gamma}_{\substack{j=1,\hdots,M \\ \norm{\gamma}_\infty \le n}} \in \C^{M\times N^2},
   \end{equation*}
   where $z_j=(x_j, y_j)=(\eim{u_j},\eim{v_j}) \in \T^2$, $\gamma=(\alpha,\beta)\in \Z^2$ is a multi-index, and $z_j^\gamma := x_j^\alpha \cdot y_j^\beta$.
   The distance of the nodes $t_j=(u_j,v_j)\in [0,1)^2$ is measured by $\abs{t_j-t_\ell}_{\T}:= \min_{r\in \Z^2}\norm{t_j-t_\ell+r}_\infty$ and we consider the situation as in Definitions \ref{def:2coll} and \ref{def:2coll2} with $K=AA^*$.
   Lemma \ref{la:btaylor} can be proven using the bivariate mean value theorem to get $|r_j|\le N\tau \pi / \abs{\xi_j}_{\T}$, $j=2,3,\hdots,M$, and the packing argument \cite[Lem.~4.5]{KuPo07} to get
   \begin{equation*}
	\norm{r}^2 
	\le (N^2- D_n(u_2)D_n(v_2))^2 + \frac{12\pi^2 N^4\tau^2}{(\rho-1)^2} \left(1+\log\ceil{\sqrt{M/6}}\right).
   \end{equation*}
   We need additional assumptions for Lemma \ref{la:2collnormKinv} to work since results for general well separated nodes, cf.~\cite{KuMoPeOh17}, seem to be too weak.
   If the nodes $t_2,\hdots,t_M$ are a subset of equispaced nodes in $\T^2$, then \cite[Cor.~4.11]{KuPo07} yields $\norm{K_2^{-1}}\le (N-N/\rho)^{-2}$.
   Together with $M\ge 4$ and $\rho\ge4+2\log M$, this yields $\norm{K^{-1}}\le 20 / N^2\tau^2$.
 \end{enumerate}
\end{remark}

\section{Pairs of nearly-colliding nodes}\label{ch:pairw}
We now study the situation in which the Vandermonde matrix comes from pairs of nearly-colliding nodes.
\begin{definition}\label{def:pairw}
Let $n\in \N$, $N = 2n+1$, $c \ge 1$ and let $t_1< \cdots < t_M \subset \intercl{0,1}$ for $M\ge 4$ even such that
\begin{align*}
	\frac{\tau}{N} &\le \abs{t_j-t_{j+\frac{M}{2}}}_\T \le \frac{c\tau}{N},\quad  j=1,\dots,\frac{M}{2},           &0&<c\tau\le 1,\\
	\frac{\rho}{N} &\le\abs{t_j-t_\ell}_\T, \quad j < \ell, \ell \ne j+\frac{M}{2},  &1&<\rho < \infty,
\end{align*}
then $\{t_1,\hdots,t_M\}$ is called a set of \emph{nodes with pairs of nearly-colliding nodes}, see Figure \ref{fig:pairw} for an illustration.
The constant $c$ measures the uniformity of the colliding nodes.
For subsequent use, we additionally introduce the following wrap around distance of indices $\abs{j-\ell}':=\min_{r\in\Z} \abs{j-\ell+r\frac{M}{2}}$ with respect to $\frac{M}{2}$.
\begin{figure}[h!]
\begin{center}
			\begin{tikzpicture}[xscale=10]
			\tikzmath{\r = 0.15; \t =0.04; \ct=\t*2; \a=0.3;
				\x1=0.07; \x5=\x1+\t;
				\x2=\x5+\r*1.2; \d1=0.5*(\x2-\x5-\r); \x6=\x2+\t*1.3; \y1=0.5*(\x6-\x2-\t);
				\x3=\x6+\r*1.6; \d2=0.5*(\x3-\x6-\r); \x7=\x3+\t*1.5; \y2=0.5*(\x7-\x3-\t);
				\x4=\x7+\r*1.2; \d3=0.5*(\x4-\x7-\r); \x8=\x4+\t*1.2; \y3=0.5*(\x8-\x4-\t);
				\d0 = 0.5*(1-\x8+\x1-\r);}
			\draw[-][draw=black, very thick] (0,0) -- (1,0);
			\draw [thick] 	(1,0)		node[below]{1} -- (1,0);
			\draw [thick] 	(0,0)		node[below]{0} -- (0,0);
			\node at (\x1,0) [label={[below,shift={(0,-\a)}] {$t_1$}}]{$\bullet$};
			\node at (\x5,0) [label={[below,shift={(0,-\a)}] {$t_5$}}]{$\bullet$};
			\draw [<->,blue](\x1,\a) 	-- (\x5,\a) node[midway,above]{$\frac{\tau}{N}$};
			\draw [|-|,red] (\x5+\d1,\a) -- (\x2-\d1,\a) node[midway,above]{$\frac{\rho}{N}$};
			\node at (\x2,0) [label={[below,shift={(0,-\a)}] {$t_2$}}]{$\bullet$};
			\node at (\x6,0) [label={[below,shift={(0,-\a)}] {$t_6$}}]{$\bullet$};
			\draw [<->,blue](\x2+\y1,\a) 	-- (\x6-\y1,\a) node[midway,above]{$\frac{\tau}{N}$};
			\draw [|-|,red] (\x6+\d2,\a) -- (\x3-\d2,\a) node[midway,above]{$\frac{\rho}{N}$};
			\node at (\x3,0) [label={[below,shift={(0,-\a)}] {$t_3$}}]{$\bullet$};
			\node at (\x7,0) [label={[below,shift={(0,-\a)}] {$t_7$}}]{$\bullet$};
			\draw [<->,blue](\x3+\y2,\a) 	-- (\x7-\y2,\a) node[midway,above]{};
			\draw [|-|,green,dashdotted](\x3,2*\a) 	-- (\x7,2*\a) node[midway,above]{$\frac{c\tau}{N}$};
			\draw [|-|,red] (\x7+\d3,\a) -- (\x4-\d3,\a) node[midway,above]{$\frac{\rho}{N}$};
			\node at (\x4,0) [label={[below,shift={(0,-\a)}] {$t_4$}}]{$\bullet$};
			\node at (\x8,0) [label={[below,shift={(0,-\a)}] {$t_8$}}]{$\bullet$};
			\draw [<->,blue](\x4+\y3,\a) 	-- (\x8-\y3,\a) node[midway,above]{$\frac{\tau}{N}$};
			\draw [|-,red] (\x8+\d0,\a) -- (1,\a) node[midway,above]{$\frac{\rho}{N}$};
			\draw [-|,red] (0,\a) -- (\x1-\d0,\a) node[midway,above]{};
			\end{tikzpicture}
\end{center}
	\caption{Example of a node set with $M=8$ satisfying Def.~\ref{def:pairw}.}\label{fig:pairw}
\end{figure}
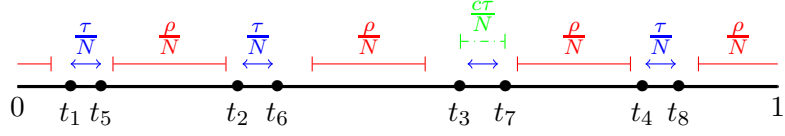
\end{definition}

\begin{definition}\label{def:decompPairwise}
	We define
	\begin{equation*} 
	 A_1:= \pmat{z_j^k}_{\substack{j=1,\dots,M/2\\ \abs{k}\le n}}\in \C^{(M/2)\times N} \quad \text{and} \quad A_2:= \pmat{z_j^k}_{\substack{j=M/2+1,\dots,M\\ \abs{k}\le n}}\in \C^{(M/2)\times N} 
	\end{equation*}
	so that with $K_1:=A_1A_1^*$, $K_2:=A_2A_2^*$ and $B:=A_2A_1^*$ we have the partitioning
	\begin{align}\label{eq:decompPairwise}
		A=\pmat{A_1\\A_2}, \quad K=\pmat{K_1 & B^* \\ B & K_2}.
	\end{align}
	Note that under the assumptions in Definition \ref{def:pairw} the Vandermonde matrices $A_1$ and $A_2$ are each corresponding to nodes that are at least $\rho/N$-separated.
\end{definition}

The proof technique we use is analogous to the one we used in the case of two nearly-colliding nodes.
The difference is that we have a matrix $K_1$ instead of a scalar and the block $B$ is a matrix instead of a vector.
Subsequently, Lemma \ref{la:pairwiseNormK} establishes an upper bound on $\norm{K}$ and Lemmata \ref{la:R1Entry}, \ref{la:restOfOrder2}, and \ref{la:pairwiseNormKInv} establish an upper bound on $\norm{K^{-1}}$.

\begin{lemma}\label{la:pairwiseNormK}
	Under the conditions of Definition \ref{def:pairw}, we have
	\begin{equation*}
		\norm{K}\le 2N \cdot\frac{\rho+1}{\rho}.
	\end{equation*}
\end{lemma}
\begin{proof}
Similar to Lemma \ref{la:2collnormK}, we start by noting that $\norm{B}^2\le \norm{K_1}\norm{K_2}$.
Together with the decomposition \eqref{eq:decompPairwise}, the triangle inequality, Lemma \ref{blockGersch}, and Theorem \ref{la:condWell}, this leads to
	\begin{align*}
		\norm{K} &\le 
		\norm{\pmat{K_1&0\\0&K_2}} +\norm{\pmat{0&B^*\\B&0}}\\
		&\le \max\set{\norm{K_1},\norm{K_2}} + \sqrt{\norm{K_1}\norm{K_2}}
		\le 2N\cdot\frac{\rho+1}{\rho}.
	\end{align*}
\end{proof}

\begin{lemma}\label{la:R1Entry}
Under the conditions of Definition \ref{def:pairw}, $R_1:= B-K_1$ fulfills
	\begin{equation*}
		\norm{R_1} \le N-\dirich{c\tau/N} + Nc\tau\left(\frac{\pi(\log\left\lfloor\frac{M}{4}\right\rfloor+1)}{\rho}+\frac{\pi^2}{6\rho^2}\right).
	\end{equation*}
\end{lemma}
\begin{proof}
	The Dirichlet kernel $\dirich{}$ is monotone decreasing on $\interc{0,1/N}$. Hence, for the diagonal entries we obtain
	\begin{equation*}
		\abs{(R_1)_{jj}} = \abs{\dirich{t_j-t_{j+\frac{M}{2}}}-N} = N -\dirich{t_j-t_{j+\frac{M}{2}}} \le N -\dirich{c\tau/N}.
	\end{equation*} 
	The off diagonal entries are bounded by the mean value theorem and Lemma \ref{la:dirichbounds} as
	\begin{align*}
		\abs{(R_1)_{j\ell}}
		&=\abs{\dirich{t_j-t_\ell}-\dirich{t_{j+\frac{M}{2}}-t_\ell}}\\
		&\le \abs{\dirichd{\xi_{j\ell}}}  \frac{c\tau}{N} 
		\le Nc\tau\left(\frac{\pi}{2 N\xi_{j\ell}} + \frac{1}{2N^2\xi_{j\ell}^2} \right),
	\end{align*}
	where $\inter{\abs{t_{j+\frac{M}{2}}-t_\ell}_\T,\abs{t_j-t_\ell}_\T}\ni \xi_{j\ell} \ge \abs{j-\ell}'\rho/N$ implies
	\begin{equation*}
		\abs{(R_1)_{j\ell}} \le Nc\tau \left(\frac{\pi}{2\rho \abs{j-\ell}'} + \frac{1}{2\rho^2(\abs{j-\ell}')^2} \right)=:(\widetilde R_1)_{j\ell}
	\end{equation*}
	for $j,\ell=1,\dots,\frac{M}{2}$, $j\ne \ell$.
	Additionally, we set $(\widetilde R_1)_{j j}:=N -\dirich{c\tau/N}$.
	We bound the spectral norm of $R_1$ by the one of the real symmetric matrix $\widetilde R_1$ using Lemma \ref{la:normAbsMat} and proceed by
	\begin{equation*}
		\norm{R_1}\le\norm{\widetilde R_1}\le\norm{\widetilde R_1}_\infty
		\le N-\dirich{c\tau/N} +  2Nc\tau  \sum_{j=1}^{\lfloor\frac{M}{4}\rfloor} \left(\frac{\pi}{2j\rho}+\frac{1}{2j^2\rho^2}\right),
	\end{equation*}
	from which the assertion follows.
\end{proof}

\begin{lemma}\label{la:restOfOrder2}
Under the conditions of Definition \ref{def:pairw}, $R_1=B-K_1$ and $R_2:= B-K_2$ fulfill
	\begin{equation*}
		\norm{2NI+R_1^*+R_2} \le 2\dirich{\tau/N} + c^2\tau^2 N\left(\frac{\pi^2(\log\lfloor\frac{M}{4}\rfloor+1)}{\rho} + \frac{\pi^3}{3\rho^2} + \frac{2.42}{\rho^3}\right).
	\end{equation*}
\end{lemma}
\begin{proof}
	First, note that
	\begin{equation*}
		(R_1^*+R_2)_{j\ell} = \dirich{t_{j+\frac{M}{2}}-t_\ell} + \dirich{t_j-t_{\ell+\frac{M}{2}}} - \dirich{t_{j+\frac{M}{2}}-t_{\ell+\frac{M}{2}}} - \dirich{t_j-t_\ell}.
	\end{equation*}
	Monotonicity of the Dirichlet kernel $\dirich{}$ on $t\in \interc{0,1/N}$ gives 
	\begin{equation*}
		\abs{(2NI+R_1^*+R_2)_{jj}}=2\abs{\dirich{t_{j+\frac{M}{2}}-t_j}} \le 2\dirich{\tau/N}
	\end{equation*}
	for $j=\ell$.
	For each fixed off diagonal entry $j\ne \ell$, the matrix $2NI$ has no contribution.
	We write the node $t_{j+M/2}$ as a perturbation of $t_j$ by $h_j:= t_{j+M/2}-t_j$ and expand the Dirichlet kernel by its Taylor polynomial of degree 2 in the point $\hat h :=t_j - t_\ell + \frac{h_j - h_\ell}{2}$. Using
	\begin{equation*}
	 D_n(h)=D_n(\hat h)+D_n(\hat h)(h-\hat h)+\frac{D_n''(\xi)}{2}(h-\hat h)^2
	\end{equation*}
        for some $\xi\in[\hat h,h]\cup[h,\hat h]$, the constant term as well as the linear term cancel out and we get
	\begin{align*}
		&D_n(t_j+h_j-t_\ell) + D_n(t_j-t_\ell-h_\ell) - D_n(t_j+h_j - t_\ell-h_\ell) - D_n(t_j-t_\ell) \\
		&=\frac{1}{8} \left( D_n''(\xi_1) (h_j+h_\ell)^2
		+ D_n''(\xi_2) (h_j+h_\ell)^2
		+ D_n''(\xi_3) (h_j-h_\ell)^2
		+ D_n''(\xi_4) (h_j-h_\ell)^2 \right).
	\end{align*}
	Lemma \ref{la:dirichbounds} and $\xi_1,\dots,\xi_4 \ge \abs{j-\ell}'\rho /N$ imply
	\begin{align*}
		\abs{(R_1^*+R_2)_{j\ell}}
		\le &\frac{N^3}{4} \left(\frac{\pi^2}{2\abs{j-\ell}'\rho}+\frac{\pi}{(\abs{j-\ell}')^2\rho^2} + \frac{1}{ (\abs{j-\ell}')^3\rho^3} \right)\\
		&\cdot \left((h_j+h_\ell)^2+(h_j-h_\ell)^2\right)
	\end{align*}
	and hence by $h_j,h_{\ell}\le c\tau/N$
	\begin{equation*}
		\abs{(2NI+R_1^*+R_2)_{j\ell}} \le N c^2\tau^2\left(\frac{\pi^2}{2\abs{j-\ell}'\rho}+\frac{\pi}{\abs{j-\ell}'^2\rho^2} + \frac{1}{ \abs{j-\ell}'^3\rho^3} \right).
	\end{equation*}
	The matrix $2NI+R_1^*+R_2$ is real symmetric so that
	\begin{align*}
		\norm{2NI+R_1^*+R_2} &\le \norm{2NI+R_1^*+R_2}_\infty \\
			&\le 2 \dirich{\tau/N} + 2\sum_{j=1}^{\lfloor\frac{M}{4}\rfloor} N c^2\tau^2\left(\frac{\pi^2}{2 j\rho}+\frac{\pi}{j^2\rho^2} + \frac{1}{j^3\rho^3} \right)\\
			&\le 2\dirich{\tau/N} + 2c^2\tau^2 N\left(\frac{\pi^2(\log\lfloor\frac{M}{4}\rfloor+1)}{2\rho} + \frac{\pi^3}{6\rho^2} + \frac{1.21}{\rho^3}\right)
	\end{align*}
	and therefore the result holds.
\end{proof}

\begin{lemma}\label{la:pairwiseNormKInv}
Under the conditions of Definition \ref{def:pairw} with $\tau\le1/2$ and $\rho\ge 2$, such that
	\begin{align*}
	\tilde C(\tau,\rho,c,M)
		&:= 2 - \frac{c^2\pi^2(\log\lfloor\frac{M}{4}\rfloor+1)}{\rho}-\frac{c^2\pi^3}{3\rho^2}-\frac{2.42c^2 }{\rho^3} \\
		&-\frac{\rho}{(\rho-1)} \left(  \frac{c^2\pi^2}{6}\tau + \frac{c\pi(\log\lfloor\frac{M}{4}\rfloor+1)}{\rho}+\frac{c\pi^2}{6\rho^2} \right)^2
	\end{align*}
	is positive, we have
	\begin{equation*}
		\norm{K^{-1}}\le \frac{C(\tau,\rho,c,M)}{N\tau^2},
	\end{equation*}
	where
	\begin{equation*}
	C(\tau,\rho,c,M):= \left( \frac{2\rho}{\rho-1}+\sqrt{\frac{\rho+1}{\rho-1}} \right)/\tilde C(\tau,\rho,c,M).
	\end{equation*}
	Figure \ref{fig:CtpcM} visualizes the values of the constant $\tilde{C}(\tau,\rho,c,M)$ with respect to $\rho$ and $\tau$.
	Please note that
	i) increasing the constant $c$ by a factor $\sqrt{2}$ has to be compensated approximately by halving $\tau$ and doubling $\rho$ and
	ii) increasing the number of nodes $M$ from $4$ to $64$ has to be compensated approximately by tripling $\rho$.
\end{lemma}

\begin{figure}[h!]
	\centering
	\begin{subfigure}{0.32\textwidth}
		\includegraphics[width=0.9\linewidth]{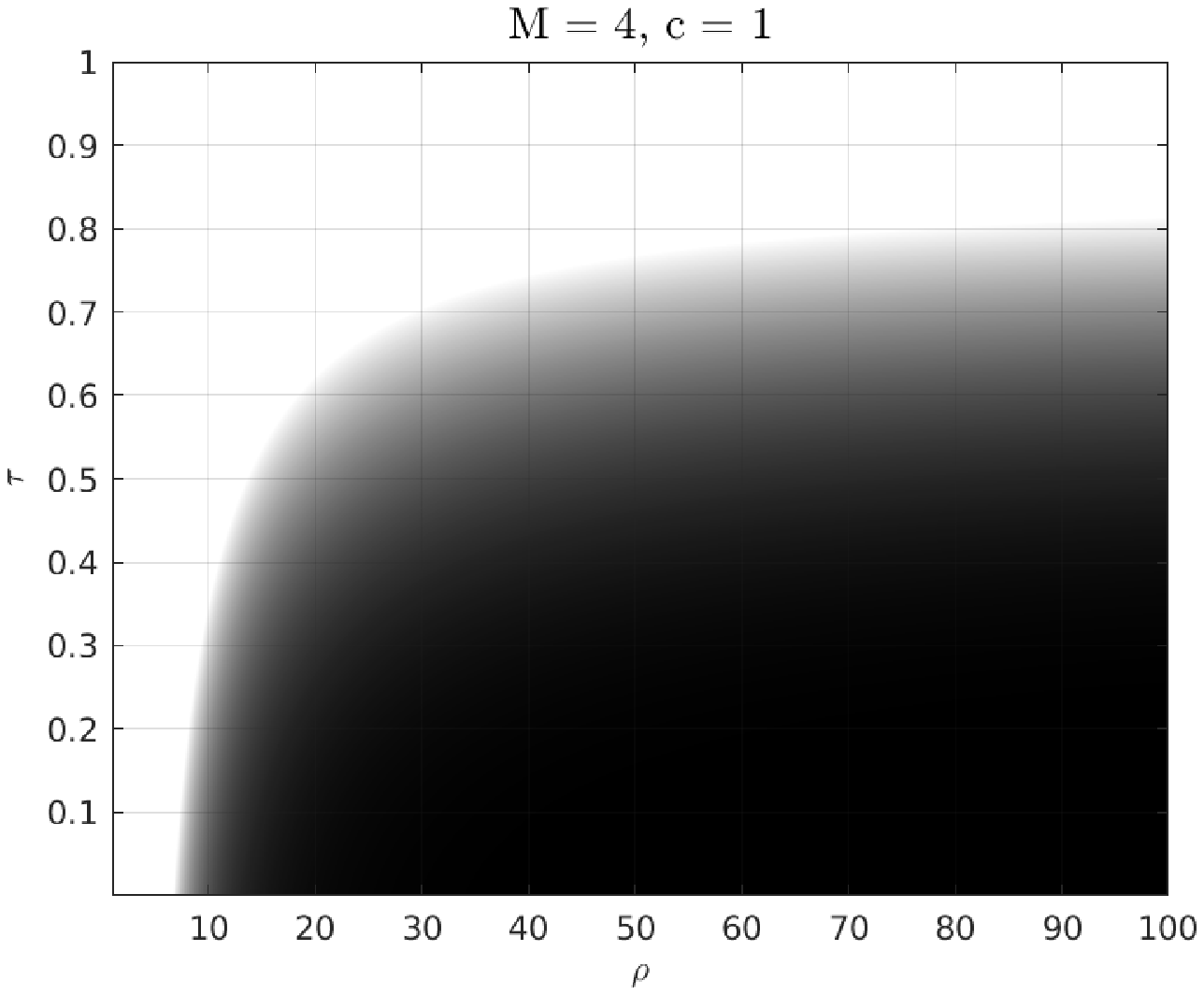}
	\end{subfigure}
	\begin{subfigure}{0.32\textwidth}
		\includegraphics[width=0.9\linewidth]{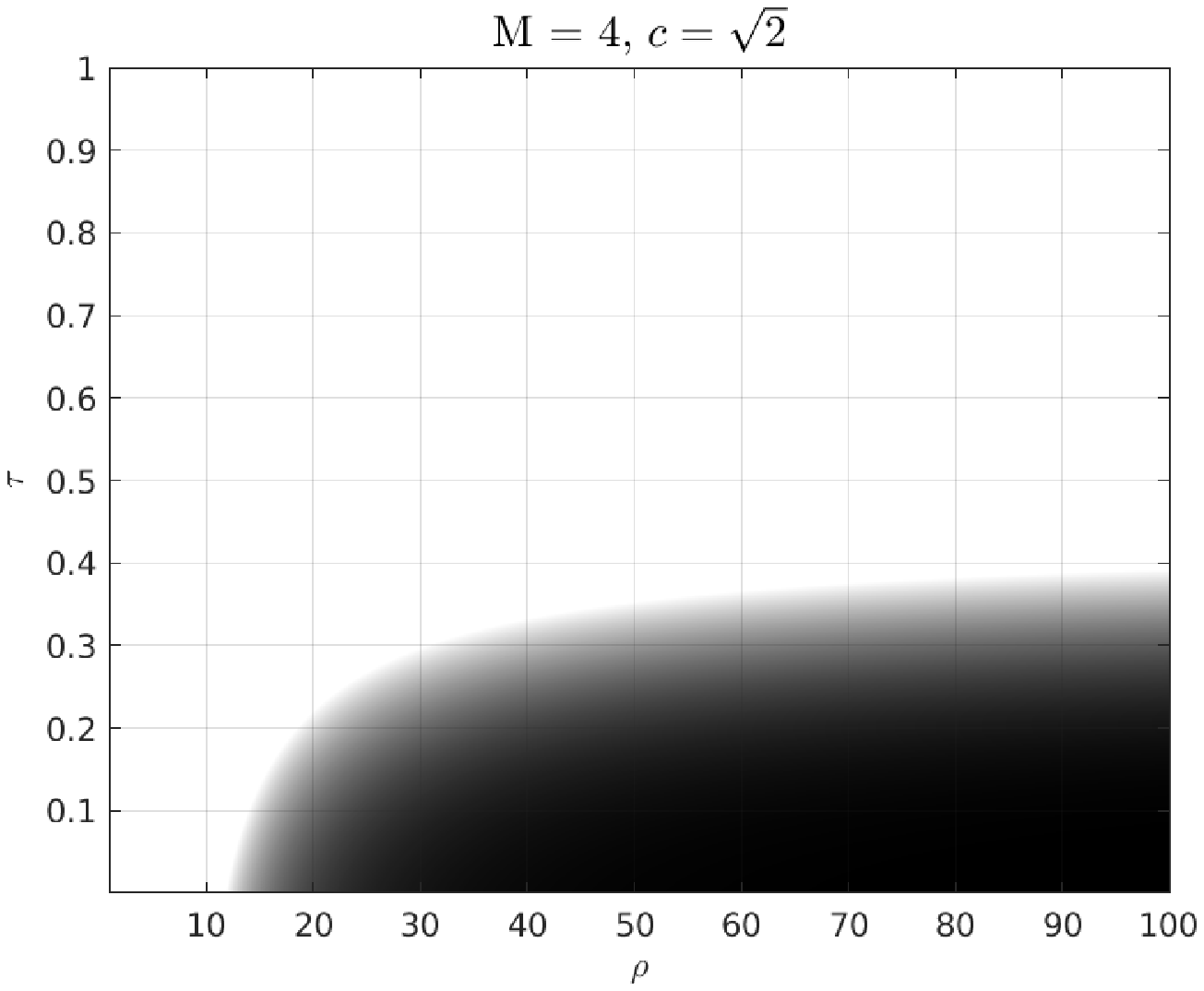}
	\end{subfigure}
	\begin{subfigure}{0.33\textwidth}
		\includegraphics[width=1.0\linewidth]{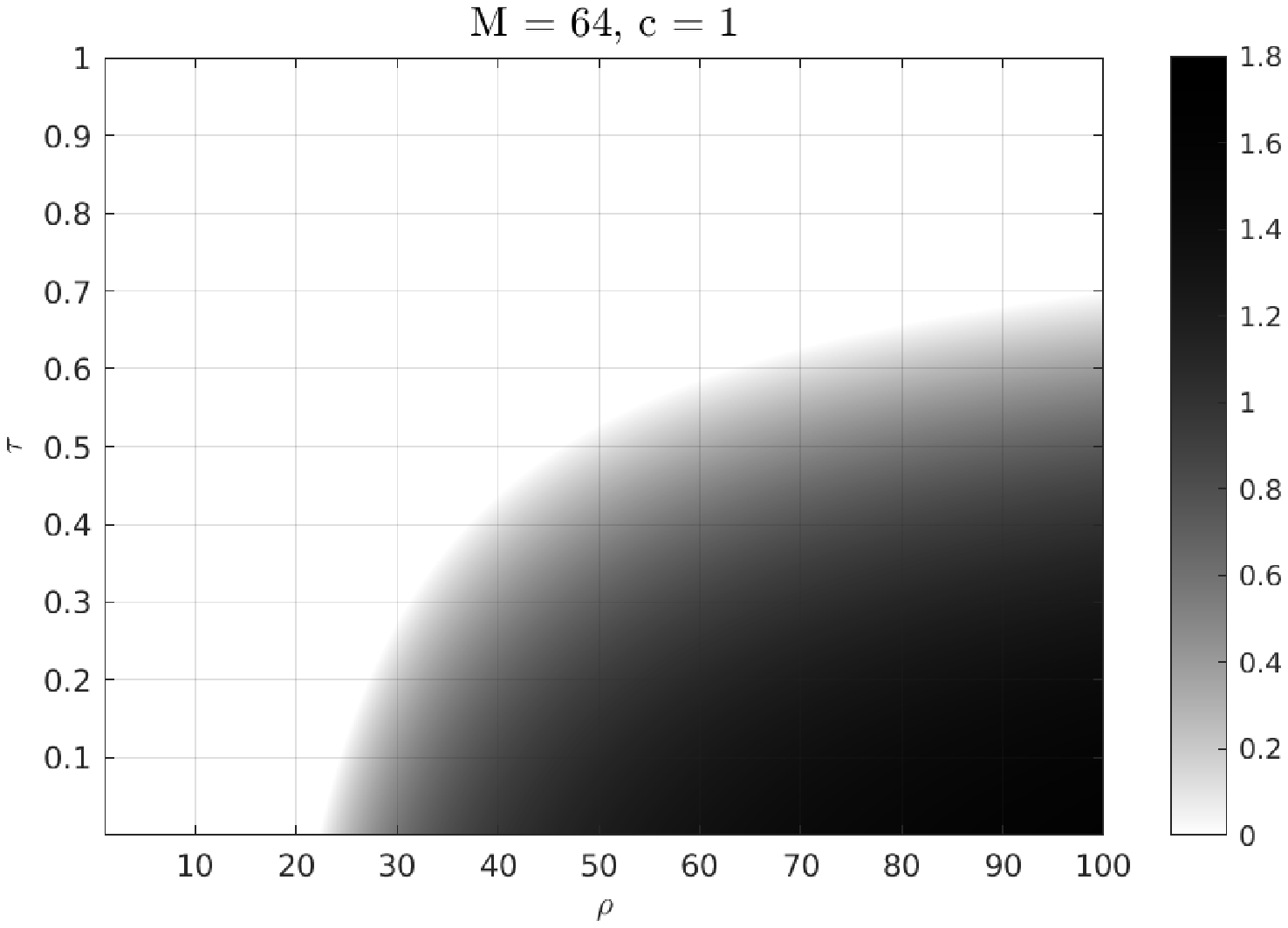}
	\end{subfigure}
	\caption{Values of $\tilde{C}(\tau,\rho,c,M)$ in Lemma \ref{la:pairwiseNormKInv} depending on $\tau$ and $\rho$ for different $M$ and $c$. Negative values are set to zero.}\label{fig:CtpcM}
\end{figure}

\begin{proof}
We proceed analogously to Lemma \ref{la:2collnormKinv} and apply Lemma \ref{la:schurDecomp} to the matrix $K$ decomposed as in $\eqref{eq:decompPairwise}$ and obtain
	\begin{align}\label{eq:pairwSchurdecomp}
		\|K^{-1}\| \le \max\{\|K_1^{-1}\|,\|(K_2 - BK_1^{-1}B^*)^{-1}\|\} \norm{\pmat{I & 0 \\ -BK_1^{-1} & I}}^2.
	\end{align}
	Definition \ref{def:decompPairwise} and Theorem \ref{la:condWell} yield
	\begin{equation*}
	 \norm{BK_1^{-1}}\le \norm{A_2}\norm{A_1^\dagger}\le\sqrt{\frac{\rho+1}{\rho-1}},
	\end{equation*}
	together with Lemma \ref{blockGersch}, we obtain
	\begin{equation*}
		\norm{\pmat{I & 0 \\ -BK_1^{-1} & I}}^2
		\le 1+\norm{BK_1^{-1}}+\norm{BK_1^{-1}}^2
		\le \frac{2\rho}{\rho-1}+\sqrt{\frac{\rho+1}{\rho-1}}.
	\end{equation*}
	
	Now, we estimate $\norm{(K_2 - BK_1^{-1}B^*)^{-1}}$, which is done by the following steps:\\
	\begin{enumerate}
	\item First, note that $I-A_1^\dagger A_1$ is an orthogonal projector and thus Theorem \ref{la:condWell} implies
		\begin{equation*}
		 \norm{K_2-BK_1^{-1}B^*} \le \norm{A_2}\norm{I-A_1^\dagger A_1}\norm{A_2^*} \le \norm{A_2}^2<2N.
		\end{equation*}
		We apply Lemma \ref{la:neumann} with $\eta=2N$, use the identities $R_1=B-K_1$ and $R_2=B-K_2$, apply the triangular inequality, and the sub-multiplicativity of the matrix norm to get
		\begin{align}\label{eq:i)}
		\begin{split}
			\norm{(K_2 - BK_1^{-1}B^*)^{-1}}
			&\le \frac{1}{2N-\norm{2NI-K_2+BK_1^{-1}B^*}}\\
			&\le \frac{1}{2N-\norm{2NI+R_1^*+R_2}-\norm{R_1}^2\norm{K_1^{-1}}}.
		\end{split}
		\end{align}
	\item Lemma \ref{la:restOfOrder2} leads to
		\begin{align*}
			2N-\norm{2NI + R_1^*+R_2}
			\ge&\; 2(N-\dirich{\tau/N})\\
			&- c^2\tau^2 N\left(\frac{\pi^2(\log\lfloor\frac{M}{4}\rfloor+1)}{\rho} + \frac{\pi^3}{3\rho^2} + \frac{2.42}{\rho^3}\right).
		\end{align*}
	\item We apply Theorem \ref{la:condWell} and Lemma \ref{la:R1Entry} to get
		\begin{align*}
			\norm{R_1}^2 \norm{K_1^{-1}}
			\le \frac{\rho}{N(\rho-1)} \left[ N-\dirich{c\tau/N} + Nc\tau\left(\frac{\pi(\log\left\lfloor\frac{M}{4}\right\rfloor+1)}{\rho}+\frac{\pi^2}{6\rho^2}\right) \right]^2.
		\end{align*}
	\item We use the estimates for the Dirichlet kernel $N-\dirich{\tau/N} \ge N\tau^2$ in ii) and $N-\dirich{c\tau/N} \le N\frac{\pi^2}{6}c^2\tau^2$ in iii), see Lemma \ref{la:dirichbounds}, and insert this in \eqref{eq:i)} to get finally
		\begin{align*}
			\norm{(K_2 - BK_1^{-1}B^*)^{-1}}
			\le&\; \frac{1}{N\tau^2} \Bigg[ 2 -
			\frac{c^2\pi^2(\log\lfloor\frac{M}{4}\rfloor+1)}{\rho}-\frac{c^2\pi^3}{3\rho^2}-\frac{2.42c^2 }{\rho^3} \\
			&-\frac{\rho}{(\rho-1)} \left(  \frac{c^2\pi^2}{6}\tau + \frac{c\pi(\log\lfloor\frac{M}{4}\rfloor+1)}{\rho}+\frac{c\pi^2}{6\rho^2} \right)^2 \Bigg]^{-1}.
 		\end{align*}
	\end{enumerate}
	
	This upper bound also bounds the maximum in \eqref{eq:pairwSchurdecomp} since for all $\tau \le 1/2$ and $\rho\ge 2$ together with Theorem \ref{la:condWell}
	\begin{equation*}
	\norm{K_1^{-1}}\le \frac{2}{N} \le \frac{1}{2N \tau^2} \le \frac{1}{N\tau^2} [2- \cdots]^{-1}.
	\end{equation*}
	
\end{proof}

\begin{thm}[Upper bound]\label{thm:pairwUp}
	Under the conditions of Definition \ref{def:pairw} with $M\ge 4$, $\tau \le \tau_{max} =\frac{1}{4c^2}$ and $\rho \ge \rho_{\min} =10c^2(\log\lfloor\frac{M}{4} \rfloor+1)$, we have
	\begin{equation*}
		\cond(A) \le \frac{5}{\tau}.
	\end{equation*}
\end{thm}
\begin{proof}
	In Lemma $\ref{la:pairwiseNormKInv}$ the constant $C(\tau,\rho,c,M)$ is monotone increasing in $\tau$ and monotone decreasing in $\rho$. Hence, after plugging in the bounds for $\tau$ and $\rho$ in our assumptions, it is easy to see that the constant $C(\frac{1}{4c^2},10c^2(\log\lfloor\frac{M}{4} \rfloor+1),c,M)$ is monotone decreasing in $c$ and $M$, respectively. Therefore, we get
	$C(\tau,\rho,c,M)\le C(1/4,10,1,4) \le 11.3$, so that $\norm{K^{-1}}\le 11.3N^{-1}\tau^{-2}$. Together with the bound $\norm{K}\le 22N/10=2.2N$ from Lemma \ref{la:pairwiseNormK}, we obtain the result.
\end{proof}

If each pair of nearly-colliding nodes has the same separation distance, i.e. $c=1$, we can improve the upper bound in the sense that restrictions on $\tau$ except for $\tau\le 1$ can be dropped.
In order to obtain the same constant, we have to increase the restrictions on $\rho$ slightly.
\begin{lemma}\label{la:pairwiseNormKInvc1}
	Under the conditions of Definition \ref{def:pairw} with $c=1$, such that
	\begin{align*}
		\tilde C(\rho,M) &:= 2 - \frac{\pi^2(\log\lfloor\frac{M}{4}\rfloor + 1)}{\rho} - \frac{\pi^3}{3\rho^2}- \frac{2.42}{\rho^3}\\
		& -\frac{\rho}{\rho-1}
		-  \frac{2\pi(\log\lfloor\frac{M}{4}\rfloor+1)}{(\rho-1)}-\frac{\pi^2}{3\rho(\rho-1)}\\
		&- \frac{\pi^2(\log\lfloor\frac{M}{4}\rfloor+1)^2}{\rho(\rho-1)} -\frac{\pi^3(\log\lfloor\frac{M}{4}\rfloor+1)}{3\rho^2(\rho-1)}- \frac{\pi^4}{36\rho^3(\rho-1)}
	\end{align*}
	is positive, we have
	\begin{equation*}
		\norm{K^{-1}}\le \frac{C(\rho,M)}{N\tau^2},
	\end{equation*}
	where $C(\rho,M):= \left( \frac{2\rho}{\rho-1}+\sqrt{\frac{\rho+1}{\rho-1}} \right) / \tilde C(\rho,M)$.
\end{lemma}
\begin{proof}
	The proof is analogous to that of Lemma \ref{la:pairwiseNormKInv}, the only difference is in step iv).
	Setting $c=1$ in ii) and iii), expanding the squared bracket in iii) and inserting this into \eqref{eq:i)} leads to
	\begin{align*}
		&\norm{(K_2 - BK_1^{-1}B^*)^{-1}}
		\le \bigg[ 2\left(N-\dirich{\tau/N}\right)\\
		&-N\tau^2 \left(\frac{\pi^2(\log\lfloor\frac{M}{4}\rfloor+1)}{\rho}+\frac{\pi^3}{3\rho^2}+\frac{2.42}{\rho^3}\right)
		- \frac{\rho}{N(\rho-1)} \left(N-\dirich{\tau/N}\right)^2	\\
		&- \frac{\rho}{\rho-1} 2\tau\left(N-\dirich{\tau/N}\right) \left(\frac{\pi(\log\lfloor\frac{M}{4}\rfloor+1)}{\rho}+\frac{\pi^2}{6\rho^2}\right)\\
		&- N\tau^2\frac{\rho}{\rho-1} \left(\frac{\pi^2(\log\lfloor\frac{M}{4}\rfloor+1)^2}{\rho^2} +\frac{\pi^3(\log\lfloor\frac{M}{4}\rfloor+1)}{3\rho^3}+ \frac{\pi^4}{36\rho^4}\right) \bigg]^{-1}.
	\end{align*}
	In three summands, we can factor out $N-\dirich{\tau/N}$ and use the estimate $N-\dirich{\tau/N}\ge N\tau^2$. Additionally, in the third summand we use the rough bound $N-\dirich{\tau/N}\le N$ and in the fourth $\tau\le1$. The same argument as in \eqref{eq:2collSchurMax} shows that this also bounds the maximum in \eqref{eq:pairwSchurdecomp} and we get the result.
\end{proof}

\begin{thm}[Upper bound]\label{thm:pairwUpc1}
	Under the conditions of Definition \ref{def:pairw} with $c=1$, $\rho \ge \rho_{\min} = 25 (\log\lfloor\frac{M}{4} \rfloor+1)$, we have
	\begin{equation*}
		\cond(A) < \frac{5}{\tau}.
	\end{equation*}
\end{thm}
\begin{proof}
	Direct inspection gives monotonicity of $C(\rho,M)$ with respect to $\rho$ and also the estimate
	 $C(25 (\log\lfloor {M}/{4} \rfloor+1),M)\le C(25,4) \le 12$.
	Hence $\norm{K^{-1}}\le {12} N^{-1} \tau^{-2}$ and together with the bound $\norm{K}\le 52N/25$ from Lemma \ref{la:pairwiseNormK} we obtain the result.
\end{proof}
 
\begin{remark}
 Due to Lemma \ref{interlacing}, the upper bound from Theorem \ref{thm:pairwUp} remains valid if nodes are removed.
 Moreover, note that $\smin$ and $\smax$ are monotone increasing with $N$ and thus, condition number estimates for an even number $N$ follow.
 
 Lower and upper bounds in Theorems \ref{thm:condlowerbound} and \ref{thm:pairwUp} yield
 \begin{equation*}
  \frac{1}{\tau}\le\cond(A)\le \frac{5}{\tau},
 \end{equation*}
 and we believe that the lower bound is quite close to what actually happens.
 The absolute constant $5$ in the upper bound follows from our proof and can be compared to the more general results from \cite{BaDeGoYo18,LiLi17}.
 Their constants grow with the total number of nodes, in \cite[Cor.~1.1]{BaDeGoYo18} quite strongly like $C M^M$ and in \cite[Thm.~1 ineq.~(2.3), Thm~2 ineq.~(2.7), and ineq.~(2.8)]{LiLi17} still like $C\sqrt{M}$, which seems artificial.
 Moreover, we would like to discuss the a-priori conditions on the parameter $N$, $\rho$, and $\tau$.
 \begin{enumerate}
 \item Our condition
  \begin{equation*}
    \rho\ge 10c^2(\log\floor{{M}/{4}}+1)
  \end{equation*}
 implies that for specific configurations of $M$ nodes, our result becomes effective as early as $N\approx CM\log M$ - this is in contrast to the results \cite{BaDeGoYo18,LiLi17}, where $N$ and $\rho$ has to be much larger.
 In \cite[Cor.~1.1, left ineq.~(10)]{BaDeGoYo18}, the conditions $N\ge 4M^3$ and $\rho\ge 2M$ are imposed.
 In \cite[Thm.~1, ineq.~(2.2), Thm~2, ineq.~(2.5)]{LiLi17} and with minor simplifications, the conditions $N\ge M^2$ and
 \begin{align}\label{eq:tauLiLiao}
  \rho\ge \begin{cases}
           C_1 \left(\frac{M}{\tau}\right)^{1/4}, & \text{with}\; C_1\approx 42,\\
           C_2 \left(\frac{M}{\tau}\right)^{1/2}, & \text{with}\; C_2\approx 63,
          \end{cases}
 \end{align}
 are imposed.
\item Our condition
 \begin{equation*}
   \tau\le \frac{1}{4c^2}
 \end{equation*}
 can be compared to the weaker condition, cf.~\cite[Cor.~1.1, right ineq.~(10)]{BaDeGoYo18},
 \begin{equation*}
   \tau\le \frac{M}{2c}.
 \end{equation*}
 Both conditions artificially involve the uniformity constant $c$.
 Except for the special cases in Theorems \ref{thm:2collup} and \ref{thm:pairwUpc1}, this prevents us from letting $\tau\to 1$.
 In contrast, \cite[Thm.~1, Thm.~2]{LiLi17} places no upper bound except $\tau\le 1$.
 However, note that condition \eqref{eq:tauLiLiao} is an a-priori lower bound on $\tau$ which prevents $\tau\to 0$ already for moderate fixed $M$, but this limitation becomes weaker for larger $N$.
 \end{enumerate}
\end{remark}

\begin{remark}(Recent other techniques)
 The manuscript \cite{Di19} considers pairs of nearly-colliding nodes and seems to weaken the assumptions considerably and might even give stronger bounds on the smallest singular value.
 The taken approach differs completely from ours and the ones in \cite{BaDeGoYo18,LiLi17}, but rather generalizes the construction of \cite{Mo15} to pairs of nearly-colliding nodes and subsets of them. 
 
 As mentioned in the introduction, the second version of \cite{LiLi17} provides a quite general framework, which can also be applied to our setting.
 The developed robust duality approach links the smallest singular value to the norm of an almost interpolating trigonometric polynomial and the interpolation error.
 The established lower bound on the smallest singular value and the condition on the separation $\rho$ depend on the number of nodes $M$ and the node separation $\tau$, such that for a fixed node set, pair nodes cannot become arbitrarily close to each other.
 Besides our generalization to the multivariate case in \cite{KuNa19}, the mentioned dependence on $M$ was dropped and the dependence on $\tau$ could be weakened considerably.
 
 In \cite{BaDiGoYo19} a QR-decomposition technique is used to establish bounds on \emph{all} singular values of Vandermonde matrices with nearly-colliding nodes.
 Adapted to the case of pairs of nearly-colliding nodes, we obtain the following.
 Let $t_1<t_2 \ll t_3<t_4$, $N=2n+1$, and $\nu:=N(t_3-t_2)$, and partition $A^*=(A_1^*\; A_2^*)$ with QR-decompositions $A_1^*=Q_1 R_1$ and $A_2^*=Q_2 R_2$, then direct computations (avoiding a so-called limit basis) yield
 \begin{equation*}
  \|Q_1^* Q_2\|_{\mathrm{F}}\le \frac{116}{\nu}.
 \end{equation*}
 Now let $M\ge 4$ even and $A$ as in Definition \ref{def:pairw}. With respect to the nearly-colliding pairs, partition $A^*=(A_1^*\; A_2^*\;\hdots\; A^*_{M/2})$ with QR-decompositions $A_j^*=Q_j R_j$ and $Q:=(Q_1\; Q_2\;\hdots\; Q_{M/2})$, and let $\rho\ge \frac{27}{23}\cdot232(\log\floor{\frac{M}{4}}+1)$, then Lemma \ref{blockGersch} yields
 \begin{equation*}
  \left|1-\lambda_{r}(Q^*Q)\right| \le \max_j \sum_{\ell=1}^{M/2} \|Q_j^* Q_{\ell}\|_{\mathrm{F}} \le 2 \sum_{\ell=1}^{\floor{M/4}} \frac{116}{\ell\rho} \le \frac{232(\log\floor{\frac{M}{4}}+1)}{\rho} \le \frac{23}{27},\quad r=1,\dots,M.
 \end{equation*}
 The Courant-Fisher min-max theorem \cite[Thm.~4.2.6]{HoJo13} and Weyl's perturbation theorem \cite[Thm.~4.3.1]{HoJo13} yield
 \begin{equation*}
  \cond(A)\le \cond(Q) \cdot \max_j\cond(A_j) \le \frac{5}{\tau}.
 \end{equation*}
 Directly following \cite{BaDiGoYo19} for the case of pairwise nearly-colliding nodes yields $\abs{(Q^* Q)_{j,k}}\le {150}/{\rho}+1079\tau$, for all $j\ne k$. Due to the uniformity of that bound and the additional summand in $\tau$, the parameter $\rho$ needs to grow at least linear in $M$.
\end{remark}

\section{Numerical Examples}

All computations were carried out using MATLAB R2019b.
As a test for the bounds in the case of one pair of nearly-colliding nodes, we use the following configuration.
Let the number of nodes $M=20$ and $M=200$ be fixed, respectively.
Moreover,  we choose $N=1+12(M-1)$ which ensures that all nodes fit on the unit interval.
We choose $\tau \in \interc{10^{-11},1}$ logarithmically uniformly at random and $\rho_3,\dots,\rho_M \in \interc{6,12}$ uniformly at random.
Then, we set the nodes $t_1<\cdots<t_M \in \intercl{0,1}$ such that $t_1=0$, $t_2=\tau/N$ and for $j=3,\dots,M$, $\abs{t_j-t_{j-1}}=\rho_j/N$.
Afterwards, the condition number of the corresponding Vandermonde matrix is computed.
This procedure is repeated $100$ times and the results are presented in Figure \ref{fig:coll_numerics} (left).

\begin{figure}[h]
	\begin{subfigure}{0.49\textwidth}
		\includegraphics[width=0.95\linewidth]{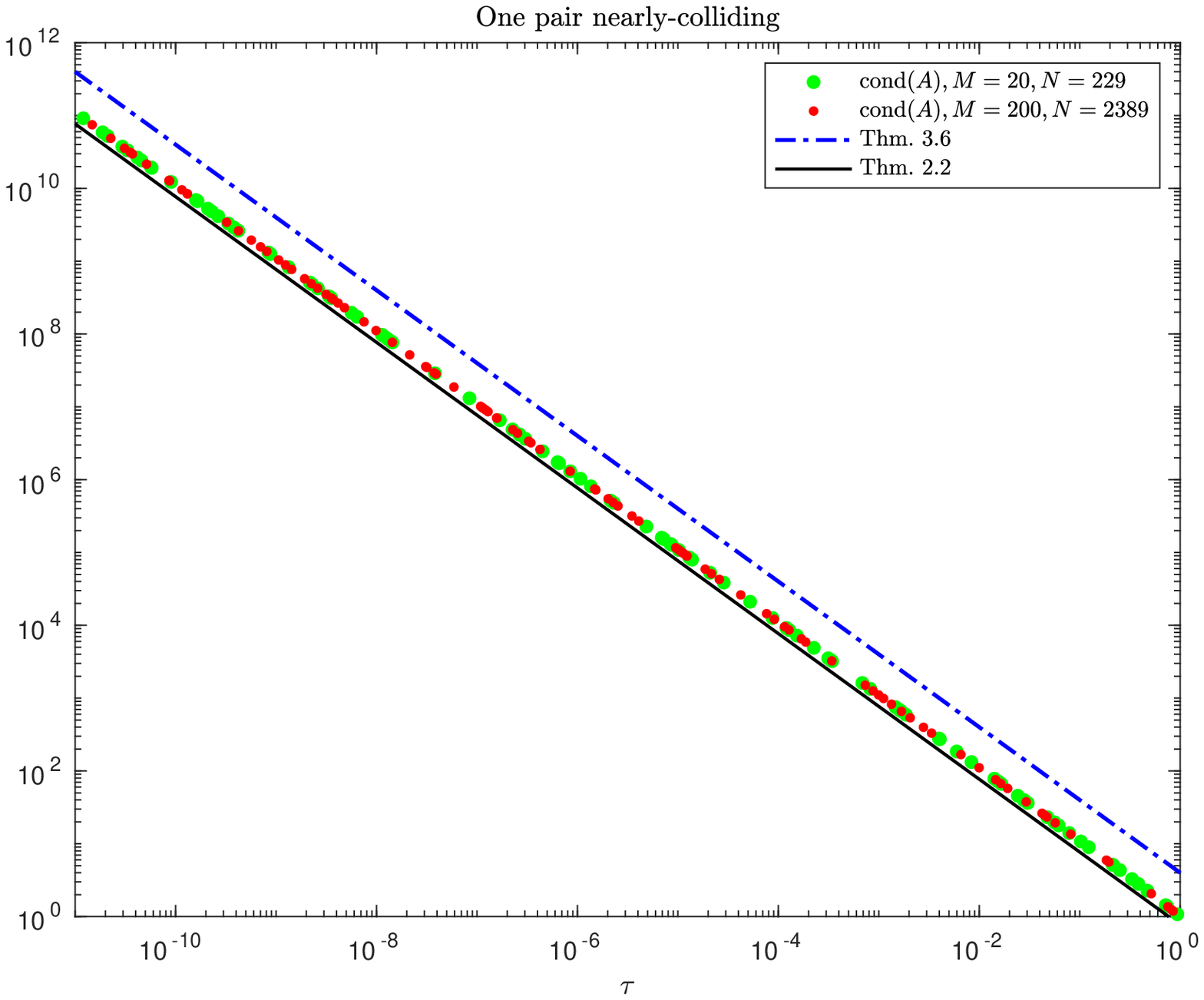}
	\end{subfigure}
	\begin{subfigure}{0.49\textwidth}
		\includegraphics[width=0.95\linewidth]{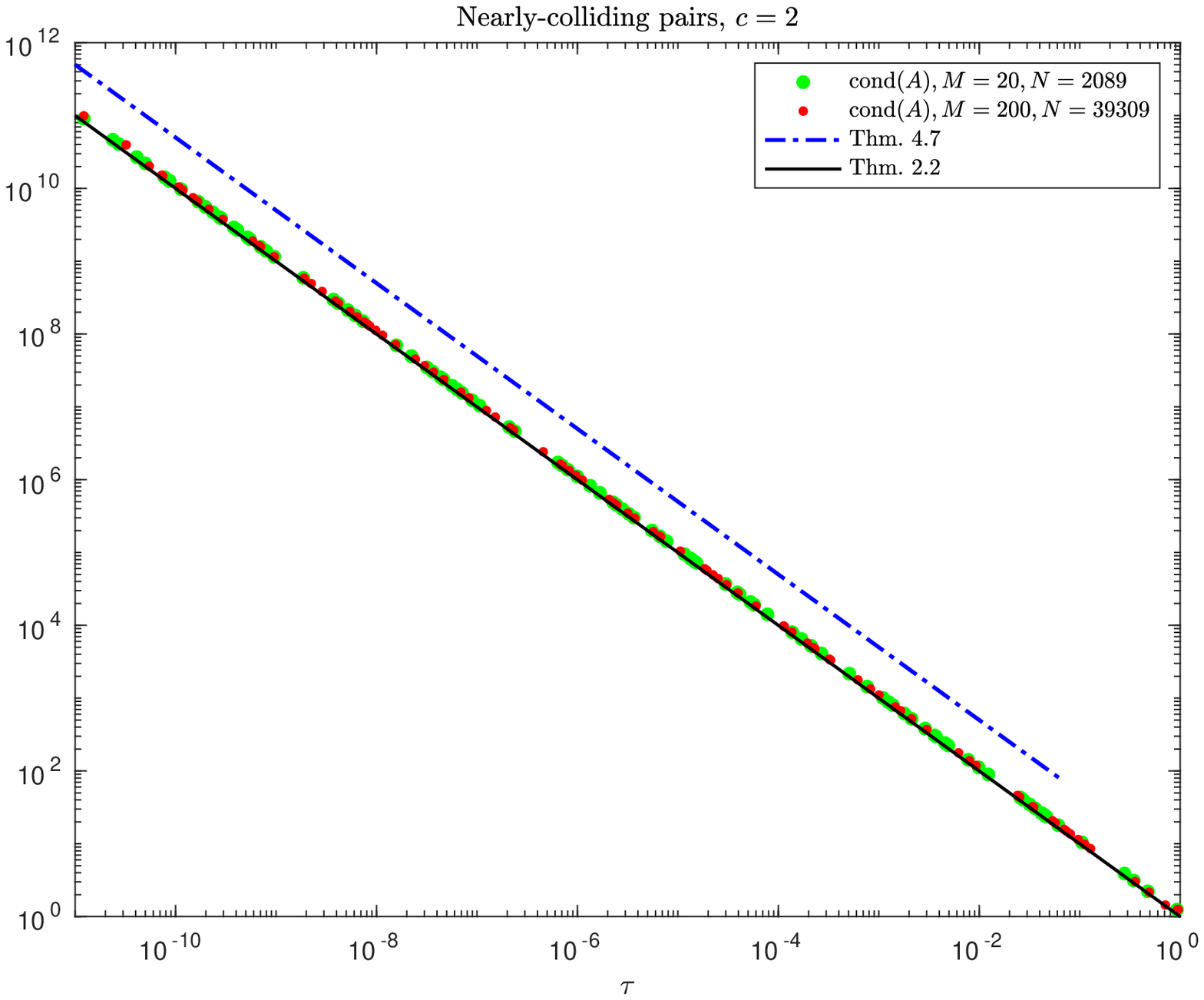}
	\end{subfigure}
	\caption{Numerical experiments for bounds on the condition number, lower bounds from Th.~\ref{thm:condlowerbound}; Left: One nearly-colliding pair, upper bound from Thm.~\ref{thm:2collup}; Right: Pairs of nearly-colliding nodes, upper bound from Thm.~\ref{thm:pairwUp}.}
	\label{fig:coll_numerics}
\end{figure}
For pairs of nearly-colliding nodes, we use the following configuration.
Let the number of nodes $M=20$ and $M=200$ be fixed, respectively.
Moreover, we choose the parameter $c=2$ and $\tau_{max}$ and $\rho_{\min}$ as in Theorem \ref{thm:pairwUp}.
To ensure that all nodes fit on the unit interval, we choose $N$ as the smallest odd integer bigger than $(c\tau_{max}+2\rho_{\min})M/2$.
Then, we choose $\tau \in \interc{10^{-11},1}$ logarithmically uniformly at random and set the nodes $t_1< \cdots < t_M \in \intercl{0,1}$ such that $t_1=0$, $t_2=\tau/N$ and for $j=3,\dots,M$, $\abs{t_j-t_{j-1}}=\rho_j/N$ if $j$ is odd or $\abs{t_j-t_{j-1}}=\tau_j/N$ if $j$ is even,
where $\tau_j\in \interc{\tau,c\tau}$ and $\rho_j \in \interc{\rho_{\min},2\rho_{\min}}$ are picked uniformly at random, respectively.
Afterwards, the condition number of the corresponding Vandermonde matrix is computed. This procedure is repeated $100$ times and the results are presented in Figure \ref{fig:coll_numerics} (right).
Note that Theorem \ref{thm:pairwUp} makes the restriction $\tau \le \tau_{max}=\frac{1}{4}$, which seems to be an artifact of our proof technique.

\begin{figure}[h]
	\begin{subfigure}{0.49\textwidth}
		\includegraphics[width=0.95\linewidth]{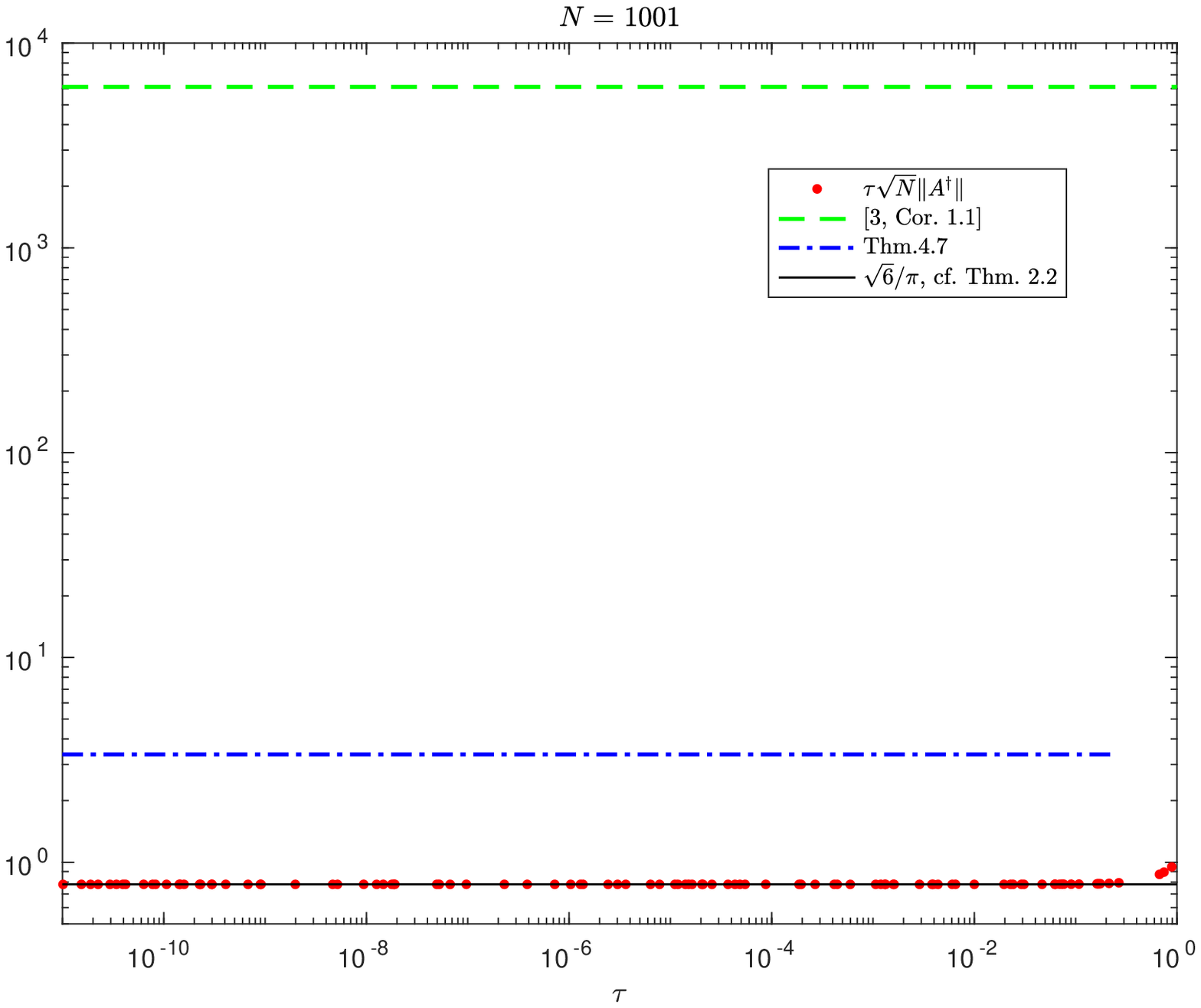}
	\end{subfigure}
	\begin{subfigure}{0.49\textwidth}
		\includegraphics[width=0.95\linewidth]{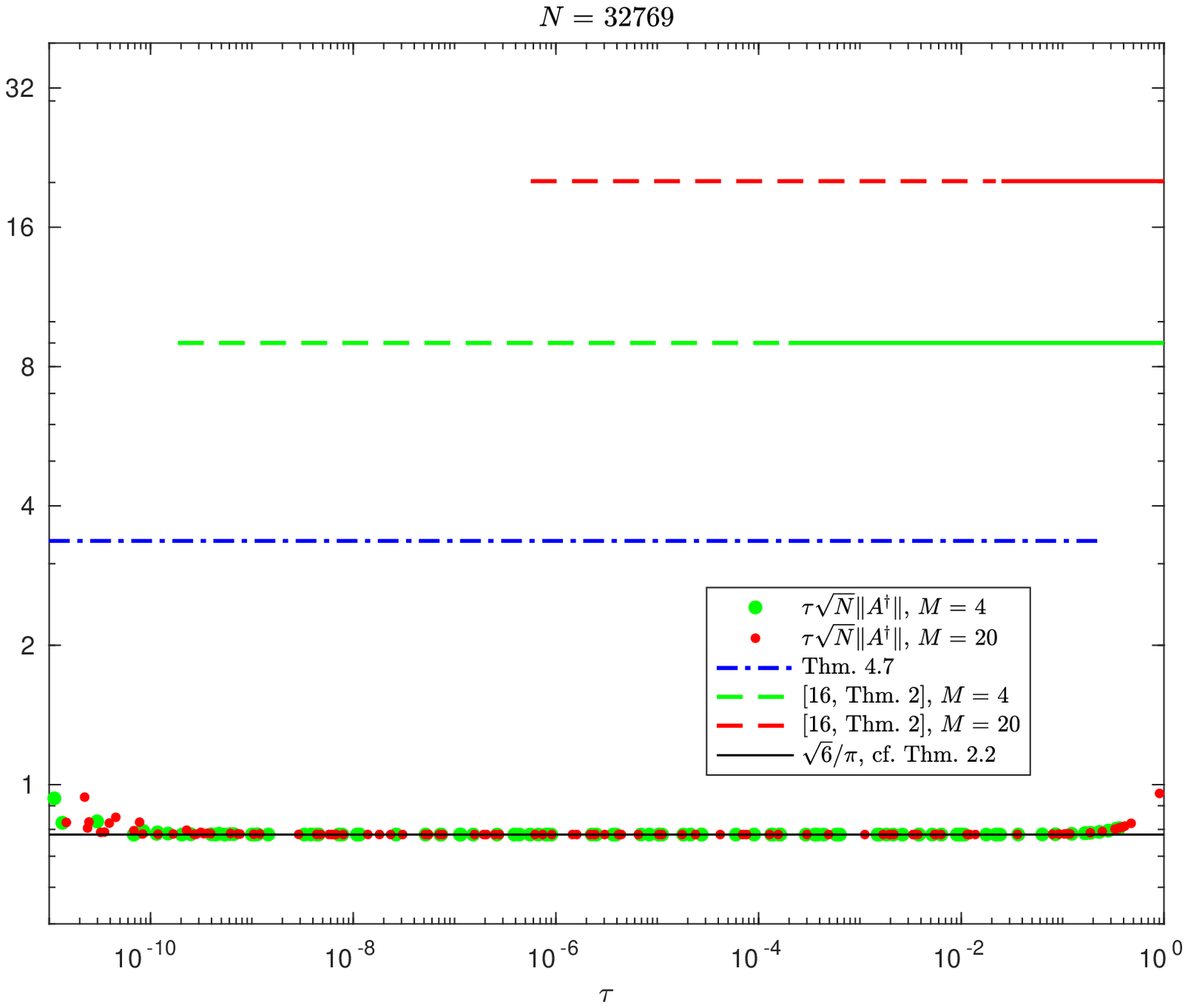}
	\end{subfigure}
	\caption{Upper bounds for $\norm{A^{\dagger}}$; Left: Comparison of Thm.~\ref{thm:pairwUp} with \cite[Cor. 1.1]{BaDeGoYo18}; Right: Comparison of Thm.~\ref{thm:pairwUp} with \cite[Thm.~2]{LiLi17}.}
	\label{fig:compare_numerics} 
\end{figure}

In order to compare Theorem \ref{thm:pairwUp} with the results from \cite[Cor.~1.1]{BaDeGoYo18}, we need to satisfy the assumptions of both results.
We take $M=3$ nodes with two nodes nearly-colliding, i.e. $t_1=0$, $t_1=\tau/N$ and $t_2=t_1+\rho/N$.
The assumptions in \cite[Cor.~1.1]{BaDeGoYo18} make it necessary that the nodes lie on an interval of length $\frac{1}{2M^2}=\frac{1}{18}$.
We choose the parameter $c=1$, $\rho_{\min} =12$, and $N=1001$.
Then, we pick $\tau \in \interc{10^{-11},1}$ logarithmically uniformly at random and $\rho \in \interc{\rho_{\min},\frac{N}{2M^2}-\tau}$ uniformly at random.
Afterwards the inverse of the smallest singular value (norm of Moore-Penrose pseudo inverse) of the corresponding Vandermonde matrix is computed.
This procedure is repeated $100$ times and the results normalized by $\tau\sqrt{N}$ are presented in Figure \ref{fig:compare_numerics} (left).
From \cite[Cor.~1.1]{BaDeGoYo18}, we get
\begin{equation*}
 \norm{A^{\dagger}} \le \frac{2(2\pi)^{M-1} M^{2M-1}}{\pi}\cdot \frac{N\sqrt{N}}{(N-1)\sqrt{N-1}}\cdot\frac{1}{\tau\sqrt{N}} \approx 6116 \cdot \frac{1}{\tau\sqrt{N}} 
\end{equation*}
for $\tau\le 1$, whereas Theorem \ref{thm:pairwUp} provides $\norm{A^{\dagger}} \le \sqrt{11.3}\cdot \frac{1}{\tau\sqrt{N}} \approx 3.4 \cdot \frac{1}{\tau\sqrt{N}}$ for $\tau \le \frac{1}{4}$.

In order to compare our results with the ones from the second version of \cite[Thm.~1, Thm.~2]{LiLi17}, we set the parameter $N=2^{15}+1$, $c=1$ and $M=4$ and $M=20$, respectively. 
All pairs of nodes are placed uniformly, such that $t_j=\frac{2j-2}{M}$ and $t_{j+M/2}=t_j+\frac{\tau}{N}$ for $j=1,\hdots,\frac{M}{2}$, where $\tau$ is picked logarithmically uniformly at random from $[10^{-11},1]$.
Afterwards, the inverse of the smallest singular value (norm of Moore-Penrose pseudo inverse) of the corresponding Vandermonde matrix is computed.
This procedure is repeated $100$ times and the results normalized by $\tau\sqrt{N}$ are presented in Figure \ref{fig:compare_numerics} (right).
Note that \cite[Thm.~1, ineq.~(2.2), Thm.~2, ineq.~(2.5)]{LiLi17} restricts \setlength{\columnsep}{-1cm}\vspace{-0.8cm}
\begin{multicols}{2}
\begin{equation*}
	\tau \ge \frac{20^2 M 2^{5} N^3}{\rho^2 (N-1)^3} \approx 
		\begin{cases}
			1.9\cdot 10^{-4},\\
			2.4 \cdot 10^{-2},
		\end{cases}
\end{equation*}
\break
\begin{equation*}
	\tau \ge \frac{10^4 2^{10} M  N^5}{\rho^4 \pi(N-1)^5} \approx 
		\begin{cases}
			1.8\cdot 10^{-10},\quad &M=4,\\
			5.6 \cdot 10^{-7},\quad &M=20,
		\end{cases}
\end{equation*}
\end{multicols}
\noindent respectively, where we used the uniform bound $\rho < \frac{2N}{M}$.
The results are shown in Figure \ref{fig:compare_numerics} (right) by proper lines \cite[Thm.~2, ineq.~(2.5)]{LiLi17} and by broken lines \cite[Thm.~1, ineq.~(2.2)]{LiLi17}.
In both cases and with minor corrections, the resulting estimate is
\begin{align*}
 \norm{A^{\dagger}}
 &\le \frac{20\sqrt{2}}{19} \left(1-\frac{\pi^2}{12}\right)^{-1/2} \frac{N-1}{2}\floor{\frac{N-1}{2}}^{-1}\frac{4}{\pi} \sqrt{M} \frac{\sqrt{N}}{\sqrt{N-1}}\cdot \frac{1}{\tau\sqrt{N}}\\
 &\approx
 \begin{cases}
   9 \cdot \frac{1}{\tau\sqrt{N}},\quad & M=4,\\
   20.1 \cdot \frac{1}{\tau\sqrt{N}},\quad &M=20,
 \end{cases}
\end{align*}
whereas Theorem \ref{thm:pairwUp} provides again $\norm{A^{\dagger}} \le 3.4 \cdot \frac{1}{\tau\sqrt{N}}$ for $\tau \le \frac{1}{4}$.
We note that our bound remains valid for $c>1$ but the restriction on $\tau$ becomes more severe.

\section{Summary}
We proved upper and lower bounds for the spectral condition number of rectangular Vandermonde matrices with nodes on the complex unit circle.
If pairs of nodes nearly-collide, the studied condition number grows linearly with the inverse separation distance.
In contrast to the more general results \cite{BaDeGoYo18,LiLi17}, we provide reasonable sharp and absolute constants but have to admit that our technique most likely will not generalize to more than two nodes nearly-colliding.
Note that our easy to achieve lower bound seems to capture the situation more accurately than the upper bound.
We posed mild technical conditions in our proofs, which cannot be confirmed to be necessary numerically.
While \cite{BaDeGoYo18} provided the right growth order for the first time, some of the imposed conditions are very restrictive and the involved constants are quite pessimistic.
The second version of \cite{LiLi17} provided a quite general framework and presented decent results with only a mild artificial growth of the condition number with respect to the number of nodes.
Moreover, a technical condition there prevents the separation distance from going to zero for a fixed number of nodes and a fixed bandwidth.
We believe that both problems can be fixed at least partially and thus \cite{LiLi17} seems to be a good framework for understanding clustered node configurations.
Recently, the manuscript \cite{Di19} came to our attention - it considers pairs of nearly-colliding nodes and seems to weaken the assumptions considerably and might even give stronger bounds on the smallest singular value.
The taken approach differs completely from ours and the ones in \cite{BaDeGoYo18,LiLi17}, but rather generalizes the construction of \cite{Mo15} to pairs of nearly-colliding nodes.

\section*{Acknowledgements}
The authors gratefully acknowledge support by the projects DFG-GK1916 and DFG-SFB944.

\bibliographystyle{abbrv}
\bibliography{references}

\goodbreak
\appendix
\section{Appendix}  
The following technical results are used within the proofs of our main results.
\begin{lemma}\label{la:dirichbounds}
	Let $n\in\N$, $N=2n+1$, then the Dirichlet kernel \eqref{eq:Dn} is bounded by
	\begin{equation*}
	N- \frac{\pi^2}{6}N^3t^2\le D_n(t)\le N - N^3 t^2,\qquad 0\le|t|\le\frac{1}{N}.
	\end{equation*}
	Furthermore, the Dirichlet kernel and its first two derivatives are bounded by
	\begin{align*}
		\abs{\dirich{t}} &\le \frac{1}{2\abs{t}},\\
		\big|\dirichd{t}\big| &\le N^2 \left(\frac{\pi}{2N|t|}+\frac{1}{2N^2|t|^2}\right),\\
		\big| \dirichdd{t}\big| &\le N^3\left(\frac{\pi^2}{2N|t|}+\frac{\pi}{N^2|t|^2}+\frac{1}{N^3|t|^3}\right)
	\end{align*}
	for $0<|t|\le 1/2$.
\end{lemma}
\begin{proof}
 Due to symmetry, it suffices to prove all bounds for $t>0$ and we use the explicit expression of the Dirichlet kernel in \eqref{eq:Dn}.
	The lower bound on $\dirich{t}$ can be derived from the inequalities $x-x^3/6\le \sin(x) \le x$, that hold for all $x\in [0,\pi]$.
	The left inequality with $x=N\pi t$ and the right inequality with $x=\pi t$ lead to
	\begin{equation*}
		\sin(N\pi t) 
		\ge \left( N- \frac{\pi^2}{6} N^3 t^2\right) \pi t
		\ge \left( N- \frac{\pi^2}{6} N^3 t^2\right) \sin(\pi t).
	\end{equation*}
	The upper bound on $\dirich{t}$ can be derived from the inequality $\cos(\alpha x) \le \cos(x)$ that holds for all $x\in [0,\pi/2]$ and $\alpha>1$ such that $\alpha x\in [0,\pi/2]$.
	Integrating this inequality, choosing $\alpha = N/2$ and $x = \pi t$, and applying the double angle formula yields
	\begin{equation*}
		\frac{\sin(N\pi t)}{2\cos(\frac{N}{2}\pi t)}=\sin\left(\frac{N}{2}\pi t\right) \le \frac{N}{2}\sin(\pi t).
	\end{equation*}
	Reordering the inequality and applying that $\cos(x)\le 1-4 x^2 /\pi^2$ for all $x\in [0,\pi/2]$ yields 
	\begin{equation*}
		\frac{\sin(N\pi t)}{\sin(\pi t)} \le N\cos(\frac{N}{2} \pi t)
		\le N(1- N^2 t^2).
	\end{equation*}
	Finally, the remaining bounds on the absolute values can be proven by calculating the first and second derivatives and using $\sin(x)\ge 2x/\pi$ and $\cot{x}\le 1/x$ that hold for all $x\in(0,\pi/2]$.
\end{proof}

\begin{lemma}\label{la:normAbsMat}
	Let $M,\widetilde M\in\C^{m\times n}$ with $\abs{M_{k\ell}}\le \widetilde M_{k\ell}$ for all $k=1,\hdots,m$, $\ell=1,\hdots,n$, then
	\begin{equation*}
	\norm{M} \le \norm{\widetilde M}.
	\end{equation*}
\end{lemma}
\begin{proof}
We directly show the result by
	\begin{align*}
		\norm{M}^2=& 	\max_{\norm{x}=1} \norm{Mx}^2 
				= 	\max_{\norm{x}=1} \sum_{k=1}^{m} \abs{\sum_{\ell=1}^{n} M_{k\ell}x_\ell}^2 
				\le \max_{\norm{x}=1} \sum_{k=1}^{m} \left(\sum_{\ell=1}^{n} \abs{M_{k\ell}}\abs{x_\ell}\right)^2\\
				\le& \max_{\norm{x}=1} \sum_{k=1}^{m} \left(\sum_{\ell=1}^{n} \widetilde M_{k\ell}\abs{x_\ell}\right)^2 = \max_{\norm{x}=1} \sum_{k=1}^{m} \left(\sum_{\ell=1}^{n} \widetilde M_{k\ell}x_\ell\right)^2 = \norm{\widetilde M}^2.
	\end{align*}
	Note that similar estimates can be found for the Frobenius norm in \cite[p.~520]{HoJo13}.
\end{proof}

\begin{lemma}[{Neumann expansion}]\label{la:neumann}
	Let $M \in \C^{n\times n}$ Hermitian and positive definite. Let $\eta \in \R$ be a parameter satisfying $\eta > \norm{M}$, then
	\begin{equation*}
	\|M^{-1}\| \leq  \frac{1}{\eta - \norm{\eta I -M}}.
	\end{equation*}
\end{lemma}
\begin{proof}
	Applying the Neumann series to the matrix $I-\eta^{-1}M$ yields
	\begin{align*}
	\norm{M^{-1}}	
	=   \frac{1}{\eta}\norm{\sum_{k=0}^{\infty}\left(I-\frac{1}{\eta}M\right)^k}
	\le \frac{1}{\eta}\frac{1}{1-\norm{I-\frac{1}{\eta}M}}
	= 	 \frac{1}{\eta-\norm{\eta I-M}}.
	\end{align*}
\end{proof}

\begin{lemma}[{Schur complement, cf.~\cite[eq.~(0.8.5.3)]{HoJo13}}]\label{la:schurDecomp}
	Let $n_1,n_2\in \N$ and the matrix $M$ $\in \C^{(n_1+n_2)\times (n_1+n_2)}$ be a $2\times 2$ block matrix of the form
	\begin{equation*}
	M=\pmat{M_1 & M_2 \\ M_3 & M_4}, M_1\in \C^{n_1\times n_1}, M_4 \in \C^{n_2\times n_2},
	\end{equation*}
	with $M_1$ being invertible. Then the Schur complement decomposition is given by
	\begin{equation*}
	M= \pmat{I_{n_1} & 0 \\ -M_3M_1^{-1} & I_{n_2}}^{-1} \pmat{M_1 & 0 \\ 0 & M_4-M_3M_1^{-1}M_2} \pmat{I_{n_1} & -M_1^{-1}M_2 \\ 0 & I_{n_2} }^{-1}.		
	\end{equation*}
	The block $[M/M_1]:=M_4-M_3M_1^{-1}M_2$ is called Schur complement of $M_1$ in $M$.
\end{lemma}

\begin{lemma}[{Cauchy interlacing theorem for eigenvalues, cf.~\cite[Thm.~(4.3.28)]{HoJo13}}]\label{interlacing}
	Let $M\in \C^{n\times n}$ be a Hermitian complex matrix, such that
	\begin{equation*}
	M = \pmat{M_1 & M_2 \\ M_2^* & M_3},\quad M_1\in\C^{m\times m}, M_2\in \C^{m\times (n-m)}, M_3\in \C^{(n-m) \times (n-m)}.
	\end{equation*}
	Let the eigenvalues of $M$ and $M_1$ be ordered in non-decreasing order, then
	\begin{equation*}
	\lambda_i(M)\le \lambda_i(M_1) \le \lambda_{i+n-m}(M),\quad i=1,\dots,m.  
	\end{equation*} 
\end{lemma}

\begin{lemma}[{Block Gerschgorin theorem, cf.~\cite[6.1.P17]{HoJo13} or \cite[Thm.~5]{FeVa62}}]\label{blockGersch}
	Let $M \in \C^{nm\times nm}$ be an $m \times m$ block matrix with blocks $M_{ik}\in \C^{n\times n}$. Let the diagonal blocks $M_{ii}$ be normal and denote $\lambda_1^{(i)},\dots, \lambda_n^{(i)}$ their eigenvalues, respectively. Then the eigenvalues of $M$ are included in the set
	\begin{equation*}
	\bigcup_{i=1}^{n}\bigcup_{j=1}^{m} \{z\in \C \colon |z-\lambda_j^{(i)}| \le \sum_{k\ne i}\|M_{ik}\|\}.
	\end{equation*}
	In particular, we have for $M\in\C^{m\times n}$ the inequalities
	\begin{align*}
	\norm{\pmat{0 & M^* \\ M & 0}} &\le \norm{M}
	\;\text{and}&
	\norm{\pmat{I & 0 \\ M & I}}^2 &\le 1+\norm{M}+\norm{M}^2.
	\end{align*}
\end{lemma}

\end{document}